\documentclass[10pt, leqno]{amsart}

\usepackage[utf8]{inputenc}

\usepackage[square, numbers, comma]{natbib} 

\usepackage[usenames,dvipsnames,svgnames,table]{xcolor}
\usepackage{amsmath,amsfonts,amsthm,amssymb,amssymb,esint,verbatim,tabularx,graphicx,stackrel}
\usepackage{enumerate}
\usepackage{fancyhdr}
\usepackage{epic}
\usepackage{pgf,tikz}
\usepackage{bbm}
\usetikzlibrary{arrows}
\usepackage{microtype}

\usepackage{color}

\usepackage{pdfpages}
\usepackage{mathtools}

\numberwithin{equation}{section}

\usepackage[hypertexnames=false]{hyperref}
\hypersetup{urlcolor=blue, colorlinks=true} 
\usepackage[inner=3cm,outer=3cm,bottom=4cm,top=4cm]{geometry}

\newtheorem{theorem}{Theorem}[section]
\newtheorem{proposition}[theorem]{Proposition}

\newtheorem{lemma}[theorem]{Lemma}
\newtheorem{corollary}[theorem]{Corollary}

\theoremstyle{definition}
\newtheorem{definition}[theorem]{Definition}

\theoremstyle{remark}
\newtheorem{remark}[theorem]{Remark}
\theoremstyle{remark}

\newcommand{\sgn}{\textup{sgn}}

\newcommand{\R}{\mathbb{R}}

\newcommand{\N}{\mathbb{N}}
\renewcommand{\L}{\mathcal{L}}
\newcommand{\B}{\mathcal{B}}

\renewcommand{\k}{\kappa}

\newcommand{\bu}{\overline{u}^c}

\renewcommand{\a}{\alpha}
\newcommand{\e}{\varepsilon}

\newcommand{\m}{\mu}
\newcommand{\cha}{\chi_\Omega}
\newcommand{\po}{\partial\Omega}

\newcommand{\dell}{\partial}
\newcommand{\norm}[2]{\|#1\|_{{#2}}}
\newcommand{\loc}{\mathrm{loc}}

\newcommand{\dd}{\,\mathrm{d}}

\DeclareMathOperator{\supp}{supp}

\DeclareMathOperator*{\esssup}{ess\,sup}
\DeclareMathOperator*{\essinf}{ess\,inf}

\newcommand{\ol}{\overline}

\begin{document}

\title[Nonlocal degenerate parabolic hyperbolic equations on bounded domains. Existence]{
Nonlocal degenerate parabolic hyperbolic equations on bounded domains. Part II: Existence}





\author[J.~Endal]{J{\o}rgen Endal}
\address[J. Endal]{Department of Mathematical Sciences\\
Norwegian University of Science and Technology (NTNU)\\
N-7491 Trondheim, Norway}
\email[]{jorgen.endal\@@{}ntnu.no}
\urladdr{http://folk.ntnu.no/jorgeen/}

\author[E.~R.~Jakobsen]{Espen R. Jakobsen}
\address[E.~R.~Jakobsen]{Department of Mathematical Sciences\\
Norwegian University of Science and Technology (NTNU)\\
N-7491 Trondheim, Norway}
\email[]{espen.jakobsen\@@{}ntnu.no}
\urladdr{http://folk.ntnu.no/erj/}

\author[O.~M{\ae}hlen]{Ola M{\ae}hlen}
\address[O.~M{\ae}hlen]{Mathematical Institute of Orsay\\
Paris-Saclay University\\
91400 Orsay, France}
\email[]{ola.maehlen\@@{}universite-paris-saclay.fr}



\begin{abstract}
We study well-posedness of degenerate mixed-type parabolic-hyperbolic equations
$$
\partial_tu+\textup{div}\big(f(u)\big)=\L[b(u)]
$$
on bounded domains with 
general Dirichlet boundary/exterior conditions. 
The nonlocal diffusion operator $\L$ is a symmetric Lévy operator (e.g. fractional Laplacians) and $b$ is nondecreasing and allowed to have degenerate regions ($b'=0$). In [N. Alibaud, J. Endal, E. R. Jakobsen, and O. Mæhlen. Nonlocal degenerate parabolic hyperbolic
equations on bounded domains. \emph{Ann. Inst. H. Poincare Anal. Non Lineaire}, 2025. Published
online first, DOI 10.4171/AIHPC/153], we introduced an entropy solution formulation for the problem and showed uniqueness of bounded entropy solutions under general assumptions. In this paper we complete the program by proving existence of such solutions.
%
%
%
 Starting from known results for scalar conservations laws, existence is proved first for bounded/zero order operators $\L$ by a fixed point argument, and then extended in steps to more general operators via approximations of $\L$.
Lack of strong compactness of approximate solutions $u_n$ is overcome through nonlinear weak-$\star$ compactness and entropy-process solutions. Key ingredients are stability results for our formulation and \textit{strong} compactness of the term $b(u)$, both with respect to variations in $\L$. 
Strong compactness follows from energy estimates and novel arguments for transferring weak regularity from $\partial_t u_n$ to $\partial_t b(u_n)$. 
Our work can be seen as both extending nonlocal theories from the whole space to domains and local theories on domains to the nonlocal case. Unlike local theories our formulation does not assume energy estimates. They are now a consequence of the formulation, but 
as opposed to previous nonlocal theories, they play an essential role in our arguments. Several results of independent interest are established, including a characterization of the Lévy operators $\L$ for which the corresponding energy/Sobolev-space compactly embeds into $L^2$.
\end{abstract}

\maketitle

\tableofcontents


\section{Introduction}
In this paper we prove existence of entropy solutions on bounded domains $\Omega\subset \R^d$ for mixed type hyperbolic-parabolic equations with nonlinear and nonlocal diffusion:
%
\begin{equation}\label{E}
\begin{cases} 
\dell_tu+\textup{div}\big(f(u)\big)=\L[b(u)] \qquad&\text{in}\qquad Q\coloneqq(0,T)\times\Omega,\\
u=u^c \qquad&\text{in}\qquad Q^c\coloneqq(0,T)\times\Omega^c,\\
u(0,\cdot)=u_0 \qquad&\text{on}\qquad \Omega,\\
\end{cases}
\end{equation}
where $T>0$, the initial/boundary (exterior) data $u_0,u^c$ are bounded, the ``fluxes" $f,b$ are (at least) locally Lipschitz and $b$ is nondecreasing and possibly degenerate,\footnote{Both $f$ and $b$ may be strongly degenerate, that is, $f'$ and $b'$ can be $0$ on whole intervals.} 
``$\textup{div}$'' is the $x$-divergence, and
the nonlocal diffusion operator $\L$, is a symmetric Lévy operator: 
Applied to $\phi \in C_c^\infty(\R^d)$ it is defined by singular integral
\begin{equation}\label{deflevy}
  \L [\phi](x):=P.V.\int_{|z|>0} \big(\phi(x+z)-\phi(x)\big) \dd \mu(z) \coloneqq\lim_{\epsilon\to0}\int_{|z|>\epsilon} \big(\phi(x+z)-\phi(x)\big) \dd\m(z),
\end{equation}
where the (Lévy) measure $\mu$ is a nonnegative and symmetric Borel measure on $\R^d$ satisfying $\mu(\{0\})=0$ and $\int_{\R^d}|z|^2\wedge 1 \dd \mu(z)<\infty$. This class of anomalous diffusion operators coincides with the generators of the symmetric pure-jump Lévy processes \cite{Ber96, Sat99, Sch03, App09}, including $\alpha$-stable, tempered, relativistic,
and compound Poisson processes. The corresponding generators include the fractional Laplacians \cite{App09,Sat99}
$$\L=-(-\Delta)^{\frac{\alpha}{2}},\quad \alpha\in(0,2),\qquad \text{where}\qquad \dd \mu_\alpha(z)\coloneqq c_{d,\a}|z|^{-(d+\alpha)}\dd z,$$ 
anisotropic operators $-\sum_{i=1}^N(-\dell_{x_ix_i}^2)^{\frac{\alpha_i}{2}}$, $\alpha_i\in(0,2)$, relativistic Schr\"odinger 
operators, 
degenerate and $0$-order operators, and numerical discretizations \cite{DTEnJa18b,DTEnJa19} of these. Since the operators $\L$ are naturally defined on the whole space, they require Dirichlet data on $\Omega^c$ 
for \eqref{E} to be well-defined. In addition, inflow conditions are needed in hyperbolic regions.

 Equation \eqref{E} is a possibly degenerate nonlinear convection-diffusion equation with applications in sciences and engineering. The solution structure is very rich and special cases include scalar conservation laws ($b=0$) \cite{MaNeRoRu96,HoRi02,Daf10}, fractional conservation laws ($b(u)=u$) \cite{BiFuWo98,DrIm06,Ali07}, nonlinear fractional diffusion equations ($f=0$), degenerate porous medium type equations \cite{DPQuRoVa11,DPQuRoVa12,DPQuRoVa14,DPQuRoVa17}, strongly degenerate Stefan type problems \cite{DTEnVa20a,DTEnVa20b}, and problems of mixed hyperbolic-parabolic type \cite{CiJa11,AlCiJa14,EnJa14}. In nondegenerate regions (where $b'(u)\neq 0$), solutions are smooth and boundary data continuously attained when $\alpha>1$. In other regions, or when $\alpha<1$, solutions may develop shocks in finite time, boundary conditions may be lost at outflow, and non-uniqueness issues arise. Kru\v{z}kov \cite{Kru70} developed an entropy solution theory and obtained well-posedness results for scalar conservation laws ($b=0$), 
a theory that was extended to problems posed on bounded domains by Bardos, Le Roux, and Nedelec \cite{BaLRNe79} and Otto \cite{Ott96,MaNeRoRu96}, see also \cite{Vov02,EyGaHe00,ChFr99,Vas01,Pan05,KwVa07}. Much later Carrillo \cite{Car99} (cf. also \cite{KaRi01,ChPe03}) came up with an entropy solution theory for local (degenerate) convection-diffusion equations ($\L=\Delta$). This first result for homogeneous Dirichlet conditions was then extended by Mascia, Poretta, and Terracina \cite{MaPoTe02} and Michel and Vovelle \cite{MiVo03}. See \cite{Val05,Kwo09, FrLi17,RoGa01, AnGa16,Kob06, Kob07,KoOh12, LiWa12, Wan16} for further developments in this direction.

Inspired by nonlocal viscosity solution theories (see e.g. \cite{JaKa06,BaIm08}), Alibaud \cite{Ali07} (linear diffusions) and Cifani and Jakobsen \cite{CiJa11} (in full generality) introduced entropy solution theories and well-posedness results for the nonlocal problem \eqref{E} posed in the whole space $\Omega=\R^d$. A key idea here, with no local analogue, is  to split the nonlocal operator in a singular and a nonsingular part and taking an additional limit.
We refer to \cite{KaUl11,AlCiJa12,AlCiJa14,CiJa14,EnJa14,IgSt18,BKV20,AAO20} and references therein for more results on Cauchy problems and further developments. To pose nonlocal problems on a bounded domain, 
the nonlocal operator must be restricted to this domain. There are many ways to do this, 
including imposing boundary conditions on the complement 
as in \eqref{E}. This is the most popular approach in PDE theory with a large literature that includes contributions from among others Caffarelli and Silvestre \cite{CaSi07,Caf12,Vaz12,R-Ot16,BuVa16}. In probability theory it corresponds to stopping or killing an underlying process upon exiting $\Omega$ \cite{Dy65}.
Other ways to restrict the operator include censoring or using spectral theory where data is only required on the boundary $\partial \Omega$ \cite{BoBuCh03,Gr18}. But the resulting operators $\L$ are no longer translation invariant, and their integral representations depend explicitly on $\Omega$. We refer to \cite{Gr18,BoFiVa18a, BoFiVa18b} for comparison of the three approaches in linear and nonlinear settings.


The literature on the nonlocal boundary value problem \eqref{E} was for some time limited to purely parabolic equations ($f\equiv0$) \cite{BoVa15, BoVa15b, BoFiVa18a, BoFiVa18b} or linear and non-degenerate diffusions ($b=\textup{Id}$) 
\cite{Bra16, Kan18, Kan20}. 
Recently, Huaroto and Neves \cite{HuNe22} showed existence (but not uniqueness)
of $L^\infty$ entropy solutions of a problem similar to \eqref{E}, but with local boundary conditions and $\L$ being the censored (or regional) fractional Laplacian of order $\alpha\in(1,2)$. This operator does not belong to the class of operators considered in this paper, and the arguments differ a lot, e.g. we work directly on the entropy formulation instead of using smooth approximations. 
In our recent companion paper \cite{AlEnJaMa23},
we studied problem \eqref{E} and gave a semi-Kru\v{z}kov type of entropy formulation that correctly accounts for exterior data and inflow boundary conditions. Our main results were uniqueness and a priori estimates for entropy solutions, results that will be summarized in the next section and play a central role in this paper. We also characterized the sense in which the boundary conditions hold using both semi-Kru\v{z}kov and Otto type formulations. Since the problem is nonlocal, a new type of boundary integrability result was needed.
%
%

 In this paper we focus on the existence of entropy solutions of \eqref{E},  complementing the uniqueness result of \cite{AlEnJaMa23}.  We show that existence holds  under mild additional assumptions on the operators $\L$, see Theorem \ref{thm: mainResultExistence}.  We also prove new compactness and stability results for \eqref{E} with respect to variations of $\L$, and obtain, as a consequence, a convergence result for the vanishing fractional viscosity method on domains.  Below, we will summarize our existence argument in four steps. Before that, we point out some difficulties on bounded domains not present in the whole space. 

When \eqref{E} is posed in $(0,T)\times \R^d$, the existence argument relies on
$L^1$-contractions \cite{Ali07, CiJa11, EnJa14, DTEnJa19}, and while the precise argument varies, the essential point is this: By translation invariance of the PDE, $L^1$-contractions imply translation estimates (first in space, then in time by Kruzkov's lemma) for the solutions; these estimates depend only on the initial data $u_0$ and not $f,b,\L$. Consequently, the solutions arising from varying $f,b,\L$ (while keeping $u_0$ fixed) constitute a precompact set in $C([0,T]; L^1_{\loc}(\R^d))$. This type of strong compactness result facilitates approximation arguments, so that existence in difficult regimes ($\L$ unbounded/singular) follows form that of the easy ones ($\L$ bounded/zero order).

On our domain $(0,T)\times \Omega$, \eqref{E} is no longer translation invariant. Another problem, is the lack of an $L^1$-contraction that incorporates the boundary/exterior data; this is actually lacking for local (degenerate) diffusion as well where, to the best of our knowledge, the only estimates akin to such $L^1$-contractions are the time-averaged results of \cite{Kob06,Kob07}. Despite these difficulties, our maximum principle (Lemma \ref{lem: maximumsPrinciple}) at least ensures that the entropy solutions of \eqref{E} arising from fixed data, form a weak-$\star$ precompact set in $L^\infty(Q)$. An alternative approach to the existence problem, based on nonlinear weak-$\star$ convergence and measure valued solutions, has been developed for hyperbolic conservation laws \cite{EyGaHe00}\footnote{The context was finite volume schemes on irregular grids which satisfy no known translation estimate in $L^1$.} and later extended to mixed hyperbolic-parabolic problems \cite{MiVo03} (see also \cite[Section 6.9]{EyGaHe00}). In \cite{MiVo03} additional energy estimates are needed to show strong $L^2$-compactness of $b(u_n)$ which is used to conclude that the limit is a solution (stability). 

We adapt the entropy-process solution concept and nonlinear weak-$\star$ approach of \cite{MiVo03} to our nonlocal setting. The details can be summarized as follows.\vspace{4pt}

\noindent\textit{The existence argument in four steps:}
\begin{enumerate}[{\rm (i)}]
    \item \textit{Existence for bounded $\L$ $\mathrm{(}$finite $\mu\mathrm{)}$}: This is Theorem \ref{thm: existenceForBoundedL}, the first result of Section \ref{sec: existenceOfEntropySolutions}. It is obtained from combining the existence result in \cite{PoVo03} with a fixed point argument. 
    \item \textit{Stability of solutions under $\L$-perturbations}: This is Proposition \ref{prop: stabilityOfEntropySolutionsWithRespectToPerturbationsInL}, the main result of Section \ref{sec: stabiltyWithRespectToVariationsInL}. It says that if Lévy operators converge 
     in an appropriate way,  $\L_n\to \L$, then the corresponding entropy solutions converge as well $u_n\to u$. The result is based on nonlinear weak-$\star$ convergence and entropy-process solutions, and it is the key to generalize the existence result to unbounded $\L$. However, the result requires that the diffusive fluxes $(b(u_n))_{n\in\N}$ are precompact in $L^2(Q)$.
    \item \textit{Regularity estimates for $b(u)$ from finite energy}: This is the subject of Section \ref{sec: precompactnessOfB(u_n)}, whose results are summarized in Corollary \ref{cor: jointCorecivityImpliesPrecompactnessOfb(u_n)}. The energy estimate (Proposition \ref{prop: finiteEnergy}) induces some translation-continuity on $b(u)$. This is used to derive a condition on $(\L_n)_{n\in\N}$ which ensures that $(b(u_n))_{n\in\N}\subset L^2(Q)$ is precompact, as needed in the previous step.
    \item \textit{Existence for unbounded $\L$ $\mathrm{(}$nonfinite $\mu\mathrm{)}$}: This is Theorem \ref{thm: existenceForUnboundedL}, the second result of Section \ref{sec: existenceOfEntropySolutions}. Under the assumption \eqref{muassumption2}, one can write $\L$ as the limit of a sequences of bounded operators $(\L_n)_{n\in\N}$ which, crucially, satisfies the condition of Corollary \ref{cor: jointCorecivityImpliesPrecompactnessOfb(u_n)}. By the stability result, and existence of entropy solutions in the bounded case, the theorem follows. 
\end{enumerate}
\begin{remark}
    Much like \cite{MiVo03}, our approximation argument also requires strong $L^2$-convergence of $b(u_n)$, but for a different reason and which is tied to the proof of Proposition \ref{prop: stabilityOfEntropySolutionsWithRespectToPerturbationsInL}; see Remark \ref{rem: whyWeNeedStrongConvergenceOfBUN} for details.
\end{remark}

 Finally, we mention  that our work can be seen as both extending
nonlocal theories from the whole space to domains and local theories on domains to the nonlocal
case. Unlike local theories our formulation does not assume energy estimates. They are now a
consequence of the formulation,  but  as opposed to previous nonlocal theories, they play an essential role in our arguments.

\subsubsection*{Outline of paper}Section \ref{sec:AssumptionsConceptMain} is reserved for assumptions, concept of solution, a priori and uniqueness results from \cite{AlEnJaMa23}, and the main result -- existence of solutions. We also give a result on convergence of the vanishing viscosity method.  In Section  \ref{sec: stabiltyWithRespectToVariationsInL} we introduce the nonlinear weak-$\star$ convergence, entropy-process solutions, and we prove stability of entropy solutions under perturbations of $\L$. In Section \ref{sec: precompactnessOfB(u_n)}, we study what translation regularity the energy estimate Proposition \ref{prop: finiteEnergy} induces on $b(u)$, culminating in Corollary \ref{cor: jointCorecivityImpliesPrecompactnessOfb(u_n)}. The main result on existence of solutions is proved in Section \ref{sec: existenceOfEntropySolutions}, along with the vanishing viscosity result. Finally, in the two appendices we show technical results needed in the proofs, including a characterization of the Lévy operators whose energy space $H^\L_0(\Omega)$ compactly embeds in $L^2(\Omega)$.


\section{Assumptions, concept of solution, auxiliary and main results.}
\label{sec:AssumptionsConceptMain}
In this section we state the notation and assumptions, present the (entropy) solution concept we will use, recall a priori and uniqueness results from \cite{AlEnJaMa23}, and give the main result on existence of solutions of \eqref{E}.  We also give a convergence result for the vanishing fractional viscosity method. 

\subsubsection*{Notation}
Let
$\phi \vee \psi\coloneqq \max\{\phi,\psi\}$, $\phi\wedge \psi \coloneqq \min\{\phi,\psi\}$, and
$\sgn^\pm(a)=\pm1$ if $\pm a >0$ and zero otherwise. In addition to the open and bounded set $\Omega\subset \R^d$, we will be working with the sets 
$$
Q=(0,T)\times\Omega,\quad \Gamma=(0,T)\times\dell\Omega, \quad Q^c=(0,T)\times\Omega^c, \quad M=(0,T)\times\R^d. 
$$
For a set $\mathcal{S}\subset \R^d$ we define the signed distance by $d_\mathcal{S}(x)=\mathrm{dist}(x,\partial S)$ if $x\in \R^d\setminus \mathcal{S}$ and $d_\mathcal{S}(x)=-\mathrm{dist}(x,\partial S)$ otherwise. For $\e>0$, the $\pm\e$-neighborhoods of $\mathcal{S}$ and $\po$ are given by
$$
\mathcal{S}_{\pm\e}\coloneqq\{x\in 
 \R^d  \,:\, d_{\mathcal{S}}(x)\leq \pm\e\}\qquad \text{and}\qquad \po_{\pm\e}:= (\po)_{\pm\e}.
$$
We will write $|\Omega|$ to mean the $d$-dimensional Lebesgue measure of $\Omega$, and $|Q|\coloneqq T|\Omega|$. The $(d-1)$-dimensional Hausdorff measure on $\po$ is denoted by $\dd \sigma(x)$. 

 The Lévy operator $\L$ is as defined in \eqref{deflevy}, and we define its corresponding bilinear operator by
\begin{align}\label{def:blform}
   \B[\phi,\psi](x) \coloneqq\lim_{\epsilon\to0}\frac{1}{2}\int_{|z|>\epsilon}  \big(\phi(x+z)-\phi(x)\big)\big(\psi(x+z)-\psi(x)\big)\dd\m(z),  
\end{align}
where the limit will typically be taken in $L^1(\R^d)$ or $L^1_{\loc}(\R^d)$. 
We denote truncated operators by
\begin{equation}\label{eq:DifferentTruncatedOperators}
    \L^{\geq r}, \qquad\L^{<r},\qquad\L^{r'> \cdots \geq r},\qquad\qquad
    \B^{\geq r}, \qquad\B^{<r},\qquad\B^{r'> \cdots \geq r},
\end{equation}
where the domains of integration are restricted to $\{|z|\geq r\}$, $\{|z|<r\}$, and $\{r'> |z|\geq r\}$; observe that $ \L^{\geq r}, \L^{<r},\L^{r'> \cdots \geq r}$ are themselves Lévy operators. 

We identify $L^2(\Omega)$ as a subspace of $L^2(\R^d)$ through zero extensions, and we further define the subspace $H^{\L}_0(\Omega)\subseteq L^2(\Omega)$ by
\begin{align}\label{eq: definitionOfLevySpaceNorm}
    \phi\in H^\L_0(\Omega) \quad \Longleftrightarrow\quad \norm{\phi}{H^{\L}_0(\Omega)}^2\coloneqq\norm{\phi}{L^2(\Omega)}^2 + \int_{\R^d}\B[\phi,\phi]\dd x<\infty.
\end{align}
A characterization of $H^\L_0(\Omega)$ on the Fourier side is given in Lemma \ref{lem: fourierCharacterizationsOfLAndB}.

The Lipschitz constant of a function $g$ on a 
set $K$ is denoted $L_{g,K}$ or simply $L_g$ if the set is clear from the context.

\subsubsection*{Assumptions} In this paper we will use the following assumptions:
\begin{align}
&\textup{$\Omega\subset \R^d$ is open, bounded with $C^2$-boundary $\dell\Omega$, and outward pointing normal $\hat{n}$}.
\tag{$\textup{A}_\Omega$}
\label{Omegaassumption}\\
&f=(f_1,f_2,\ldots,f_d)\in W_{\textup{loc}}^{1,\infty}(\mathbb{R};\mathbb{R}^d).
\tag{$\textup{A}_f$}
\label{fassumption}\\
&\text{$b\in W_{\textup{loc}}^{1,\infty}(\R;\R)$ is non-decreasing, and the weak derivative $ b'\in{ B V_{\textup{loc}}(\R)}$.}
\tag{$\textup{A}_b$}
\label{bassumption}\\
\tag{$\textup{A}_{u^c}$}&u^c\in (C^2\cap L^\infty)(Q^c)\textup{ and has an extension $\overline{u}^c\in (C^2\cap L^\infty)([0,T]\times\R^d)$}.
\label{u^cassumption}\\
&u_0\in L^{\infty}(\Omega).
\label{u_0assumption}
\tag{$\textup{A}_{u_0}$}\\
&\textup{$\mu\geq 0$ is a Borel measure on $\R^d$, symmetric about zero $\mu(\cdot)=\mu(-\cdot)$,}\label{muassumption}
\tag{$\textup{A}_{\mu}$}\\ &\textup{no point-mass at zero $\mu(\{0\})=0$, and satisfies $\textstyle\int_{\R^d}\big(|z|^2\wedge 1\big)\dd \mu(z)<\infty$.}\notag
\\
&\textup{$\mu$ is such that the multiplier  $m(\xi)\coloneqq\textstyle\int_{\R^d}\big(1-\cos(\xi\cdot z)\big)\dd\mu(z)\to\infty$ \ as \ $|\xi|\to\infty$.}
\label{muassumption2}
\tag{$\textup{A}_{\mu}'$}
\end{align}

\begin{remark}\label{assumptionremark}\leavevmode
\begin{enumerate}[{\rm (a)}]
\item  By a symmetric Lévy operator $\L$, we will always mean an operator of the form \eqref{deflevy} whose (Lévy) measure $\mu$ satisfies \eqref{muassumption}. Note also that while $\mu$ need not be locally bounded at zero, it is $\sigma$-finite, and so the upcoming product measures will be unambiguous.
\item In \eqref{fassumption} and \eqref{bassumption}, we can assume without loss of generality that $f(0)=0$ and $b(0)=0$ (add constants to $f$ and $b$), and that $f$ and $b$ are globally Lipschitz and $b'$ has bounded total variation (solutions are uniformly bounded by Lemma \ref{lem: maximumsPrinciple}). 
\item In \eqref{bassumption} 
the condition $b'\in  B V_{\textup{loc}}(\R)$ implies $b\in W^{1,\infty}_{\textup{loc}}(\R)$, the standard assumption for the Cauchy problem. Our stronger assumption still allows for classical power type and strongly degenerate nonlinearities  $b(r)=r^m$ for  $m>1$ (porous medium) and  $b(r)=\max\{r-L,0\}$ 
(one-phase Stefan). 
\item Assumption \eqref{muassumption2} is a necessary and sufficient condition for a family of functions bounded in the energy space $H_0^{\L}$ to be compact in $L^2$, see Proposition \ref{prop: coerciveSymbolEqualsCompactEmbeddingOfLevySpaceInL2}. It is needed in the existence proof in Section \ref{sec: existenceOfEntropySolutions} where it compensates for the lack of $L^1$-contraction results. We refer to $m$ as the (Fourier) multiplier of $\L$, though, strictly speaking, it is the multiplier of $-\L$ (Lemma \ref{lem: fourierCharacterizationsOfLAndB}). 
\end{enumerate}
\end{remark}
\subsubsection*{Entropy solutions} For the definition of entropy solutions we introduce the semi-Kru\v{z}kov entropy-entropy flux pairs
\begin{equation*}
(u-k)^\pm,\qquad
F^{\pm}(u,k)\coloneqq\sgn^\pm(u-k)(f(u)-f(k))
\qquad  \text{for}\qquad u,k\in\mathbb{R},
\end{equation*}
where $(\cdot)^+\coloneqq\max\{\cdot,0\}$ and $(\cdot)^-\coloneqq(-\cdot)^+$, and the
usual splitting of the nonlocal operator,
\begin{equation*}
\L[\phi](x)=\L^{<r}[\phi](x)+\L^{\geq r}[\phi](x) \qquad\text{for}\qquad \phi\in C_c^{\infty}(\R^d),\ r>0,\ x\in \mathbb{R}^d,
\end{equation*}
where $\L^{<r}$ and $\L^{\geq r}$ are defined in \eqref{eq:DifferentTruncatedOperators}.
Our definition is then:

\begin{definition}[Entropy solution]\label{def: entropySolutionVovelleMethod}
A function $u\in L^\infty(M)$ is an \emph{entropy solution} of \eqref{E}
if:
\begin{enumerate}[(a)]
\item\label{def: entropyInequalityVovelleMethod-1} \textup{(Entropy inequalities in $\overline Q$)} For all $r>0$, and all $k\in\R$ and $0\leq \varphi\in C^\infty_c([0,T)\times\R^d)$ satisfying 
\begin{align}\label{eq: admissibilityConditionOnConstantAndTestfunction}
(b(u^c)-b(k))^\pm\varphi = 0\qquad \text{in $Q^c$},
\end{align}
the following inequality holds
\begin{equation}\label{eq: entropyInequalityVovelleMethod}
\begin{split}
&\,-\int_Q \Big( (u-k)^\pm\dell_t\varphi + F^\pm(u,k)\cdot\nabla\varphi \Big)\dd x\dd t\\
&\,-\int_Q\L^{\geq r}[b(u)]\sgn^\pm(u-k)\varphi\dd x\dd t - \int_M(b(u)-b(k))^\pm\L^{< r}[\varphi]\dd x\dd t\\
\leq&\,\int_\Omega (u_0-k)^\pm\varphi(0,\cdot) \dd x + L_f\int_\Gamma (\overline{u}^{c}-k)^\pm\varphi \dd\sigma(x)\dd t.
\end{split}
\end{equation}
\item\label{def: entropyInequalityVovelleMethod-2} \textup{(Data in $Q^c$)} $u=u^c$ a.e.~in $Q^c$.
\end{enumerate}
\end{definition}
This definition is an extension of both \cite[Definition 2.1]{CiJa11} (see also \cite{Ali07}) for nonlocal problems in the whole space and \cite[Definition 2.1]{MiVo03} (see also \cite{MaPoTe02}) for local problems on bounded domains. In the hyperbolic case ($b'\equiv0$), it is equivalent \cite{Vov02} to the original definition of Otto \cite{Ott96}. By \eqref{Omegaassumption}--\eqref{muassumption}, all the integrals in \eqref{eq: entropyInequalityVovelleMethod} are well-defined.
\begin{remark}\label{rem: defentsolnremark}\leavevmode
\begin{enumerate}[(a)]
\item The condition $(b(u^c)-b(k))^+\varphi = 0$ is weaker than $|b(u^c)-b(k)|\varphi = 0$. 
When $b(u^c)< b(k)$, the first condition holds for any $\varphi$ while the second then implies that $\varphi=0$.
This gives a hint to why standard Kru\v{z}kov entropy-entropy flux pairs are too restrictive in this setting, see \cite{Vov02} for counterexamples to uniqueness in the hyperbolic case when $b=0$.
\item \label{rem: defentsolnremark(b)}  The truncated Lévy operator $\L^{\geq r}$ is bounded on $L^p(\R^d)$ for all $p\in[1,\infty]$. It can be written $\mathcal{L}^{\geq r}[b(u)]=\mu_r\ast b(u)-\mu_r(\R^d) b(u)$ where $\mu_r$ is the \textit{finite} Borel measure $\mu_r(E)\coloneqq \mu(E\setminus \{|z|<r\})$. Note that the convolution term is well defined because $b(u)\mapsto \mu_r\ast b(u)$ maps all Borel representatives of $b(u)$ to the same element in $L^\infty(M)$; see \cite[Proposition 8.49]{Fo:book} for details.

\end{enumerate}
\end{remark}

\subsubsection*{A priori results and uniqueness} In \cite{AlEnJaMa23} we adapt Kruzkov type arguments to prove that $L^\infty$-bounds, energy estimates, and uniqueness hold for  entropy  solutions according to this definition.

\begin{lemma}[$L^\infty$-bound \cite{AlEnJaMa23}]\label{lem: maximumsPrinciple}
Assume \eqref{Omegaassumption}--\eqref{muassumption} and  $u$ is an entropy solution of \eqref{E}. Then
\begin{align*}
   \min\Big\{\essinf_{\Omega} u_0, \essinf_{Q^c} u^c\Big\} \leq u(t,x)\leq  \max\Big\{\esssup_{\Omega} u_0,  \esssup_{ Q^c} u^c\Big\}\qquad\text{for a.e. $(t,x)\in Q$.}
\end{align*}
\end{lemma}



Recall that $\B$ is the bilinear operator associated with $\L$ defined in \eqref{def:blform}, and define
$$
F(u,k)\coloneqq (F^++F^-)(u,k)=\sgn(u-k)(f(u)-f(k)).
$$
\begin{proposition}[Energy estimate \cite{AlEnJaMa23}]\label{prop: finiteEnergy}
Assume \eqref{Omegaassumption}--\eqref{muassumption} and $u$ is an entropy solution of \eqref{E}.
Then 
\begin{align}\label{eq: finiteEnergy}\nonumber
&\,\int_M \B\big[b(u)-b(\bu),b(u)-b(\bu)\big] \dd x\dd t\\
    \leq &\,\int_{\Omega} H(u_0,\bu(0))\dd x - \int_{Q} \Big[(u-\bu) \bu_t + F (u,\bu)\cdot\nabla\bu\Big]b'(\bu)\dd x\dd t \\
   &\,+ \int_{Q} \L[b(\bu)](b(u)-b(\bu))\dd x\dd t,
    \nonumber
\end{align}
where  
$
H(u,k):=\int_k^u \big(b(\xi)-b(k)\big)\dd \xi \geq 0$. Moreover, the right-hand side of \eqref{eq: finiteEnergy} is finite.
\end{proposition}

\begin{remark}\label{remark: energyEstimateImpliesLocalBoundedEnergyOfBU}
This global energy estimate for $b(u)-b(\bu)$ implies 
a local energy estimate for $b(u)$: By the inequality $x^2\leq 2(x-y)^2 + 2y^2$ we have
$$
\int_M \B\big[b(u),b(u)\big]\varphi \dd x\dd t\leq 2\int_M \B\big[b(u)-b(\bu),b(u)-b(\bu)\big]\varphi \dd x\dd t+2\int_M \B\big[b(\bu),b(\bu)\big] \varphi\dd x\dd t
$$
for any $0\leq \varphi \in C_c^\infty(M)$, where 
$\B\big[b(\bu),b(\bu)\big]$ is locally bounded since $b(\bu)$ is bounded and locally Lipschitz. 
If $\int_M \B[b(\bu),b(\bu)] \dd x\dd t<\infty$, we can send $\varphi\to 1$ to get a global energy estimate for $b(u)$.
\end{remark}

\begin{theorem}[Uniqueness \cite{AlEnJaMa23}]\label{thm: uniquenessEntropySolutions}
Assume \eqref{Omegaassumption}--\eqref{muassumption}. If $u$ and $v$ are entropy solutions of \eqref{E}
with data $u_0,v_0$, and $u^c=v^c$, then for a.e. $t\in (0,T)$
\begin{align*}
    \int_{\Omega}|u(t,x)-v(t,x)|\dd x \leq \int_{\Omega}|u_0(x)-v_0(x)|\dd x.
\end{align*}
In particular, if also $u_0=v_0$, then $u=v$ a.e. in $(0,T)\times\R^d$.
\end{theorem}

\subsubsection*{Existence and vanishing viscosity}

The main contribution of this paper is the following existence result for \eqref{E} under slightly stronger assumptions on $\L$ (or $\mu$).  The result is a combination of Theorems \ref{thm: existenceForBoundedL}  and \ref{thm: existenceForUnboundedL},  and Corollary \ref{cor: allAbsolutelyContinuousMeasuresAreIncludedInOurExistenceResult}. A summary of its proof was given at the end of the previous section.

\begin{theorem}[Existence]\label{thm: mainResultExistence}
Assume \eqref{Omegaassumption}--\eqref{muassumption}.  If either $\mu(\R^d)<\infty$ or \eqref{muassumption2} holds, then there exists an entropy solution $u$ of \eqref{E}. In particular, if $\mu$ is absolutely continuous with respect to the Lebesgue measure on $\R^d$, then there exists an entropy solution $u$ of \eqref{E}. 
%
\end{theorem}
\begin{remark}
     The conditions $\mu(\R^d)<\infty$ and \eqref{muassumption2} are mutually exclusive (Lemma \ref{lem: zeroOrderLevyOperatorsHaveBoundedSymbols}) and represent two manageable cases: In the first, $\L$ is bounded on $L^2(\R^d)$, and in the other, $\L$ is `singular enough' so that $(\textup{Id}+\L)^{-1}\colon L^2(\Omega)\to L^2(\R^d)$ is compact. The remaining intermediate case, where both $\mu(\R^d)<\infty$ and \eqref{muassumption2} fail, cannot be covered by the techniques here, but we expect that existence still holds. These problematic operators $\L$ are characterized in terms of their multiplier $m$ (Lemma \ref{lem: examplesOfBadLevyOperators}), but not easily described directly in terms of $\mu$: It is necessary (Lemmas \ref{lem: zeroOrderLevyOperatorsHaveBoundedSymbols}, \ref{lem: absolutelyContinuousMeasureHasNiceSymbol}) that $\mu_a(\R^d)<\mu_s(\R^d)=\infty$, where $\mu_a,\mu_s$ are the absolutely continuous and singular part of $\mu$, but this is not a sufficient condition (see second example in Lemma \ref{lem: examplesOfBadLevyOperators}).
\end{remark}

As a simple byproduct of our approach,  we obtain a convergence result for the fractional-nonlinear vanishing viscosity method, 
\begin{equation}\label{E3}
\begin{cases} 
\dell_tu_n+\textup{div}\big(f(u_n)\big)=-\frac1n(-\Delta)^{\frac{\alpha}{2}}[b(u_n)] \qquad&\text{in}\qquad Q,\\
u_n=u^c \qquad&\text{in}\qquad Q^c,\\
u_n(0,\cdot)=u_0 \qquad&\text{on}\qquad \Omega,
\end{cases}
\end{equation}
where $\alpha\in(0,2)$.
Solutions $u_n$ converge as $n\to\infty$ to the solution $u$ of the boundary value problem for the scalar conservation law 
\begin{equation}\label{eq:SCL}
\begin{cases} 
\dell_tu+\textup{div}\big(f(u)\big)=0 \qquad&\text{in}\qquad Q,\\
u=u^c \qquad&\text{on}\qquad \Gamma,\\
u(0,\cdot)=u_0 \qquad&\text{on}\qquad \Omega.\\
\end{cases}
\end{equation}

\begin{proposition}[Vanishing viscosity]\label{prop:vvisc}
Assume \eqref{Omegaassumption}--\eqref{u_0assumption}
and $u$ and $u_n$ are entropy solutions of \eqref{eq:SCL} and \eqref{E3}, respectively. Then
$$
u_n\to u \quad \text{as}\quad n\to\infty\qquad\text{in\quad $L^p(Q)$}\quad \text{for all $p\in[1,\infty)$.}
$$
\end{proposition}
 In fact, in the previous result one can replace $-(-\Delta)^{\frac{\alpha}{2}}$ with any $\L$ covered by Theorem \ref{thm: mainResultExistence}; the below proof is identical.
\begin{proof}[Proof of Proposition \ref{prop:vvisc}]
The proof follows the same steps as that of Proposition \ref{prop: stabilityOfEntropySolutionsWithRespectToPerturbationsInL}. While we lack precompactness of $(b(u_n))_{n\in\N}$, this is not needed as all terms featuring $\L$ (or, its truncations) are easily seen to vanish.
\end{proof}



\section{Nonlinear weak-$\star$ convergence, entropy-process solutions, and stability results} \label{sec: stabiltyWithRespectToVariationsInL}
 In this section we prove a stability results for entropy solutions of \eqref{E} under perturbations of $\L$; this is Proposition \ref{prop: stabilityOfEntropySolutionsWithRespectToPerturbationsInL}. The main tools shall be the notion of nonlinear weak-$\star$ convergence and entropy-process solutions   
following \cite[Section 6.9]{EyGaHe00}. 

\begin{definition}[Nonlinear weak-$\star$ convergence]\label{def: nonlinearWeakStartConvergence}
Let $U$ be an open subset of $\R^{1+d}$, $(u_n)_{n\in\N}\subset L^\infty(U)$, and $u\in L^\infty(U\times(0,1))$. Then $(u_n)_{n\in\N}$ converges towards $u$ \emph{in the nonlinear weak-$\star$ $L^\infty(U)$ sense} if 
\begin{equation}\label{eq: propertyOfYoungMeasure}
\begin{split}
    &\,\lim_{n\to\infty} \int_U \psi(u_n(t,x))\varphi(t,x) \dd x\dd t\\
    = &\, \int_0^1\int_U \psi(u(t,x,a))\varphi(t,x) \dd x \dd t \dd a, \qquad\text{for all $\varphi\in L^1(U)$ and all $\psi\in C(\R)$.}
\end{split}
\end{equation}
\end{definition}

\begin{remark}\label{remark: weakConvergence}\leavevmode
\begin{enumerate}[{\rm (a)}]
\item When $u_n\to u$ in the nonlinear weak-$\star$ $L^\infty(U)$ sense, then  $$
u_n\to \int_0^1u(\cdot,\cdot,a)\dd a
$$
in the (classically) weak-$\star$ $L^\infty(U)$ sense, and the sequence $(u_n)_n$ is bounded in $L^\infty(U)$. 
\item Any bounded sequence in $L^\infty(U)$ has a subsequence converging in the nonlinear weak-$\star$ $L^\infty(U)$ sense. See Proposition 6.4 in \cite{EyGaHe00} for a proof.
\item If a nonlinear weak-$\star$ limit $u$ of $(u_n)_n$ does not depend on $a$, i.e. $u\in L^\infty(U)$, then $u_n\to u$ in $L^p_{\textup{loc}}(U)$ for every $p\in[1,\infty)$. See Remark 6.16 in \cite{EyGaHe00} for a proof.
\item Nonlinear weak-$\star$ convergence is equivalent to Young measure convergence of \cite{DPe85}.
\end{enumerate}
\end{remark}

This notion of convergence is particularly useful for stability results in scalar parabolic-hyperbolic conservation laws, and for two. Firstly, $L^\infty$-control on a sequence of solutions is typically a trivial matter by maximum principles. Secondly, one can often upgrade nonlinear weak-$\star$ convergence to \textit{strong} convergence by a clever (now standard) argument: The first step of this argument is to define an appropriate solution concept for these nonlinear weak-$\star$ limits, and this leads to so-called entropy-\textit{process} solutions,  introduced in different settings e.g. in 
\cite[Eq. (5.36), p. 139]{EyGaHe00}, \cite[Def. 3]{Vov02}, and \cite[Def. 3.1]{MiVo03}. This concept allows for a family of functions $u_a\in L^\infty(M)$, for $a\in(0,1)$, to be considered a solution of \eqref{E} if it satisfies an $a$-averaged version of the entropy inequalities \eqref{eq: entropyInequalityVovelleMethod}.

\begin{definition}[Entropy-process solution]\label{def: entropyProcessSolution}
We say that $u=u(t,x,a)\in L^\infty(M\times (0,1))$ is an \emph{entropy-process solution} of \eqref{E} if:
\begin{enumerate}[(a)]
\item\label{def: entropyProcessInequalityVovelleMethod-1} \textup{(Averaged entropy inequalities in $\overline{Q}$)} For all $r>0$, and all $k\in\R$ and $0\leq \varphi\in C^\infty_c([0,T)\times\R^d)$ satisfying 
\begin{equation*}
(b(u^c)-b(k))^\pm\varphi = 0\qquad \text{a.e. in $Q^c$},
\end{equation*}
we have the inequality
\begin{equation}\label{eq: entropyProcessInequalityVovelleMethod}
\begin{split}
&\,-\int_0^1\int_Q \Big( (u-k)^\pm\partial_t\varphi + F^\pm(u,k)\cdot\nabla\varphi \Big)\dd x\dd t\dd a\\
&\,-\int_0^1\int_M \Big(\L^{\geq r}[b(u)-b(k)]\sgn^\pm(u-k)\varphi\cha + (b(u)-b(k))^\pm\L^{<r}[\varphi]\Big)\dd x\dd t\dd a\\
\leq&\,\int_\Omega (u_0-k)^\pm\varphi(0,\cdot) \dd x + L_f\int_\Gamma (\overline{u}^{c}-k)^\pm\varphi \dd\sigma(x)\dd t.
\end{split}
\end{equation}
\item\label{def: entropyProcessInequalityVovelleMethod-2} \textup{(Data in $Q^c$)} $u(t,x,a)=u^c(t,x)$ for a.e. $(t,x,a)\in (0,T)\times \Omega^c\times(0,1)$.
\smallskip
\item \textup{(Non-degenerate points)} The expression $b(u(t,x,a))$ is essentially independent of $a$ in $M\times (0,1)$.
\end{enumerate}
\end{definition}

\begin{remark}\label{rem:es}
 By (c) we see that $u$ is automatically constant in $a$ whenever $b'(u)>0$. In particular, if $u$ is independent of $a$ for a.e. $(t,x)\in Q$, then $\tilde{u}(t,x)\coloneqq \int_0^1u(t,x,a)\dd a$ is the unique entropy solution of \eqref{E}, as can be seen by  comparing Definition \ref{def: entropySolutionVovelleMethod} with the above.
\end{remark}

Clearly, every entropy solution $u$ corresponds to an entropy-process solution $\tilde u$ by defining $\tilde u(t,x,a)\coloneqq u(t,x)$. Under our assumptions
the correspondence also holds the other way:

\begin{theorem}[Entropy-process solutions are entropy solutions]\label{thm: uniquenessEntropyProcessSolutions}
Assume \eqref{Omegaassumption}--\eqref{muassumption} and let $u$ be an entropy-process solution of \eqref{E}. Then $u(t,x,a)$ is 
independent of $a$  a.e.,  and 
$$
\tilde{u}(t,x)\coloneqq \int_0^1 u(t,x,a)\, \dd a
$$
is the unique entropy solution of \eqref{E}.
\end{theorem}
    
\begin{proof}
We first prove uniqueness of entropy-process solutions: Let $u$ and $v$ be entropy-process solutions of \eqref{E} with initial data $u_0=v_0$ and exterior data $u^c=v^c$. By our definition, entropy-process solutions satisfy the same a priori estimates as those given for entropy solutions in \cite[Section 2]{AlEnJaMa23}. The proofs follow the \textit{exact} same steps, after integrating all expressions over $a\in (0,1)$ (expressions independent of $a$ are unaffected). Consequently, if we double variables setting $u=u(t,x,a)$ and $v=v(s,y,\tilde{a})$, 
we can follow the steps of the uniqueness proof of \cite[Section 5]{AlEnJaMa23} (integrating over $a,\tilde{a}\in(0,1)$) to conclude that
\begin{align}\label{proc_dbl}
    \int_0^1\int_0^1\int_\Omega |u(t,x,a)-v(t,x,\tilde a)|\dd x \dd a\dd \tilde a\leq  0,
\end{align}
for a.e. $t\in (0,T)$. This inequality forces the integrand to be essentially independent of $(a,\tilde a)\in (0,1)^2$, which can only hold if both $u$ and $v$ are essentially independent of $a,\tilde a$. Thus, they both coincide with $\tilde u(t,x)\coloneqq \int_0^1u(t,x,a)\dd a$ almost everywhere in $Q\times (0,1)$. As explained in Remark \ref{rem:es}, $\tilde u$ is then necessarily the unique entropy solution of \eqref{E}. 
\end{proof}

With these tools at hand, we can address the main problem of this section: Assume that $(u_n)_{n\in\N}$ is a sequence of entropy solutions of
\begin{equation}\label{E2}
\begin{cases} 
\dell_tu+\textup{div}\big(f(u)\big)=\L_n[b(u)] \qquad&\text{in}\qquad Q,\\
u=u^c \qquad&\text{in}\qquad Q^c,\\
u(0,\cdot)=u_0 \qquad&\text{on}\qquad \Omega.
\end{cases}
\end{equation}
 If $\L_n\to\L$, then we expect under certain assumptions that $u_n\to u$ where $u$ is the solution of the limiting problem \eqref{E}. The type of convergence for $\L_n$ we consider here, is a weighted total variation norm of the corresponding measures:
\begin{align}
\int_{\R^d} \big(|z|^2\wedge1\big)\ \dd|\mu_n-\mu|(z) \to 0
\qquad \text{as}\qquad n\to\infty,
\label{muassumptionlim}
\tag{$\textup{A}_{\mu}$-lim}
\end{align}
where $|\nu|=\nu^++\nu^-$ is the total variation measure of a signed measure $\nu$. Under this convergence, we have for $\phi\in C_c^\infty(\R^d)$, and $p\in[1,\infty]$:
\begin{align}\label{eq:convL}
&\|\L\phi_n\|_p\leq \, 2M\|\phi\|_{W^{2,p}},
\quad \text{and}\quad 
\|\L_n \phi - \L\phi\|_p\leq 2\|\phi\|_{W^{2,p}}\int_{\R^d} \big(|z|^2\wedge1\big)\ \dd|\mu_n-\mu|(z).
\end{align}
where $M\coloneqq \sup_n\int_{\R^d}(|z|^2\wedge1)\dd \mu_n(z)$ is finite by triangle argument using \eqref{muassumptionlim} and that $\mu$ satisfies \eqref{muassumption}. These $L^p$ bounds follow from \cite[Lemma 3.4]{AlEnJaMa23}. 
\begin{remark}\label{rem: definitionOfVagueConvergence}
    Defining $\tilde{\mu}_n$ by $\dd \tilde\mu_n (z)\coloneqq (|z|^2\wedge 1)\dd\mu_n(z)$, we see that \eqref{muassumptionlim} means $\tilde\mu_n\to \tilde\mu$ in total variation norm.  A natural weaker assumption is   
    $\tilde\mu_n\to\tilde \mu$ in the vague (weak-$\star$) sense which, by a version of Lévy's continuity theorem, is equivalent to pointwise convergence of the corresponding multipliers $m_n(\xi)\to m(\xi)$. The stronger convergence concept \eqref{muassumptionlim} allows  for shorter and less technical proofs, 
but it can be shown that the stability result below holds also under vague convergence of $\tilde\mu_n$. 
\end{remark}
\begin{proposition}[Stability of entropy solutions w.r.t. $\L$]\label{prop: stabilityOfEntropySolutionsWithRespectToPerturbationsInL}
 Assume \eqref{Omegaassumption}--\eqref{muassumption} and $(u_n)_{n\in \N}$ are solutions of \eqref{E} for corresponding Lévy operators $(\L_n)_{n\in \N}$, all with the same data $u_0,u^c$. If $\L_n\to \L$ in the sense of \eqref{muassumptionlim} and $(b(u_n))_{n\in \N}$ is precompact in $L^2(Q)$, then $u_n \to u$ in $L^p(Q)$ for any $p\in [1,\infty)$, where $u$ is the unique solution of \eqref{E} for the corresponding Lévy operator $\L$. 
\end{proposition}


\begin{proof}
 Note first that the sequence $(u_n)_{n\in \N}$ is uniformly bounded in $L^\infty(Q)$ because of Lemma \ref{lem: maximumsPrinciple} and the fact that $u_0$ and $u^c$ are independent of $n$. By Remark \ref{remark: weakConvergence} and the precompactness of $(b(u_n))_{n}\subset L^2(Q)$ there is a subsequence, again denoted by $(u_n)_{n\in \N}$, which converges in nonlinear weak-$\star$ sense to some $u=u(x,t,a)$ and such that $b(u_n)$ converges strongly. Since $b$ is continuous, the compatibility of these two convergences imply that $b(u)$ is  a.e.  
independent of $a\in (0,1)$ and 
\begin{equation}\label{eq: strongL2Convergence}
    \lim_{n\to\infty}\|b(u_n(\cdot,\cdot))- b(u(\cdot,\cdot,a))\|_{L^2(Q)}=0, \quad \text{for a.e. }a\in(0,1).
\end{equation}
Now, suppose $u$ is an entropy-process solution of \eqref{E}. Then it follows from Theorem \ref{thm: uniquenessEntropyProcessSolutions} that $u$ is  a.e.  
independent of $a\in (0,1)$, and so by canonically associating $u$ with $\int_{0}^1u\,\dd a$ we get both that $u_n\to u$ in $L^p(Q)$ for $p\in[1,\infty)$ (see Remark \ref{remark: weakConvergence}) and that $u$ is the unique entropy solution of \eqref{E}. Finally, because this argument can be carried out for any subsequence of \textit{the original} sequence $(u_n)_{n\in\N}$ it follows by uniqueness of the limit that the original sequence converges as well. Thus, the proof is complete if we can show that $u$ is an entropy-process solution of \eqref{E}, which we will now do.

Let $k\in\R$, $r>0$, and $0\leq \varphi\in C^\infty_c(M)$ satisfy
 \begin{align}\label{eq: conditionKPhiForExistenceOfProcessSolution}
     (b(u^c)-b(k))^\pm\varphi=0,\qquad\text{a.e. in $Q^c$.}
 \end{align} 
Since $u_n$ is an entropy solution of \eqref{E2}, we have
\begin{equation}\label{eq: entropyFormulationForCensoredLevyOperator}
    \begin{split}
        &\,-\int_Q\Big( (u_n-k)^\pm\partial_t\varphi + F^\pm(u_n,k)\cdot\nabla\varphi\Big)\dd x \dd t\\
        &\,- \int_Q \sgn^{\pm}(u_n-k)\L_n^{\geq r}[b(u_n)]\varphi \dd x\dd t - \int_M(b(u_n)-b(k))^\pm\L_n^{<r}[\varphi]\dd x\dd t\\
    \leq &\, \int_{\Omega}(u_0-k)^{\pm}\varphi(0,\cdot)\dd x + L_f\int_{\Gamma}(\overline{u}^c - k)^{\pm}\varphi \dd \sigma(x)\dd t.
    \end{split}
\end{equation}
  We now let $n\to\infty$ in \eqref{eq: entropyFormulationForCensoredLevyOperator}. Note that \eqref{eq: conditionKPhiForExistenceOfProcessSolution} and the two last terms in \eqref{eq: entropyFormulationForCensoredLevyOperator}  are independent of $n$. Also, the first integral in \eqref{eq: entropyFormulationForCensoredLevyOperator} converges to the corresponding one in \eqref{eq: entropyProcessInequalityVovelleMethod} by nonlinear weak-$\star$ convergence of $u_n$ since $(r-k)^\pm$ and $F^\pm(r,k)$ 
are continuous. 
The third integral in \eqref{eq: entropyFormulationForCensoredLevyOperator} converges to $$-\int_0^1\int_M(b(u)-b(k))^\pm\L^{<r}[\varphi]\dd x\dd t\dd a$$ because of the strong $L^2$-convergence of both $b(u_n)$ and $\L^{< r}_n[\varphi]$ from \eqref{eq: strongL2Convergence}, \eqref{muassumptionlim}, and \eqref{eq:convL}.

It remains to establish the limit of $\int_Q \sgn^{\pm}(u_n-k)\L_n^{\geq r}[b(u_n)]\varphi \dd x\dd t$. For a moment we think of $\sgn^\pm$ as a continuous function. By adding and subtracting terms and applying the triangle inequality,
\begin{equation}\label{eq:ProblematicNonlocalTermInNAndEpsilon}
\begin{split}
&\bigg|\int_M \sgn^{\pm}(u_n-k)\L_n^{\geq r}[b(u_n)]\varphi\cha \dd x\dd t-\int_0^1\int_M \sgn^{\pm}(u-k)\L^{\geq r}[b(u)]\varphi\cha \dd x\dd t\dd a\bigg|\\
&\leq \, \int_M|\sgn^{\pm}(u_n-k)|\times\\
&\qquad\quad\times\Big( \Big|\L_n^{\geq r}[b(u_n)]-\L^{\geq r}[b(u_n)]\Big|+\Big|\L^{\geq r}[b(u_n)]-\int_0^1\L^{\geq r}[b(u)]\dd a\Big|\Big)\varphi\cha \dd x\dd t\\
&\quad+\bigg|\int_0^1\int_M\big(\sgn^\pm(u_n-k)-\sgn^\pm(u-k)\big)\L^{\geq r}[b(u)]\varphi\cha\dd x \dd t\bigg|.
\end{split}
\end{equation}
The last term goes to zero as $n\to\infty$ by nonlinear weak-$\star$ convergence since $\L^{\geq r}[b(u)]\varphi\cha\in L^1(M).
$ By the Cauchy-Schwarz inequality, and the inequality\footnote{This inequality follows from the triangle and Minkowski integral inequalities (cf. the proof of Lemma 2.1 in \cite{EJ21}).}
$$\|(\L_n^{\geq r}-\L^{\geq r})[\psi]\|_{L^2}\leq 2\|\psi\|_{L^2} \int_{|z|\geq r}\dd|\mu_n-\mu|(z),$$
the other term is bounded by
\begin{align*}
&\Big(\big\|\big(\L_n^{\geq r}[b(u_n)]-\L^{\geq r}[b(u_n)]\big)\big\|_{L^2(M)}+\Big\|\L^{\geq r}\Big[b(u_n)-\int_0^1b(u)\dd a\Big]\Big\|_{L^2(M)}\Big)\|\varphi\|_{L^\infty}\norm{\cha}{L^2(M)}\\
&\leq \, 2\|\varphi\|_{L^\infty}\norm{\cha}{L^2}\bigg(\|b(u_n)\|_{L^2}\int_{|z|\geq r}\dd|\mu_n-\mu|(z) +\Big\|b(u_n)-\int_0^1b(u)\dd a\Big\|_{L^2}\int_{|z|\geq r}\dd\mu_\alpha(z)\bigg),
\end{align*}
where $\norm{\cha}{L^2(M)}=(T|\Omega|)^{\frac{1}{2}}$. 
These terms converge to zero by \eqref{muassumptionlim} and since $b(u_n)$ converges (and is bounded) in $L^2$. 

Since $\sgn^\pm(\cdot-k)$ is not a continuous function, we need an approximation argument. Note first that if the $\pm$ entropy inequalities \eqref{eq: entropyFormulationForCensoredLevyOperator} is satisfied for $(\varphi,k)$ satisfying \eqref{eq: conditionKPhiForExistenceOfProcessSolution}, then it also satisfied for $(\varphi,\xi)$ for any $\pm\xi\geq\pm k$.
Hence integrating from $\xi=k$ to $\xi=k+\e$ in the $(+)$ case and $\xi=k-\e$ to $\xi=k$ in the $(-)$ case and dividing by $\e$, we find that
 \begin{equation}\label{eq: mollifiedEntropyFormulationForCensoredLevyOperator}
    \begin{split}
        &\,-\int_Q (u_n-k)_\e^\pm\partial_t\varphi + F^\pm_\e(u_n,k)\nabla\varphi\dd x \dd t\\
        &\,- \int_Q \sgn^{\pm}_\e(u_n-k) \L_n^{\geq r}[b(u_n)]\varphi \dd x\dd t - \int_M(b(u_n)-b(k))_\e^\pm\L_n^{<r}[\varphi]\dd x\dd t\\
    \leq &\, \int_{\Omega}(u_0-k)_\e^{\pm}\varphi(0,\cdot)\dd x + L_f\int_{\Gamma}(\overline{u}^c - k)_\e^{\pm}\varphi \dd \sigma(x)\dd t,
    \end{split}
\end{equation}
where 
 $g_\e^+(k)= \frac{1}{\e}\int_k^{k+\e}g^+(\xi)\dd\xi$ and
 $g_\e^-(k)= \frac{1}{\e}\int_{k-\varepsilon}^{k}g^-(\xi)\dd\xi$ for $g^\pm(k)=(a-k)^\pm$, $\sgn^{\pm}(a-k)$, and $F^\pm(a,k)$.
Note that these functions are all continuous.
Fixing $\varepsilon>0$, we can now argue as above and send $n\to\infty$ in \eqref{eq: mollifiedEntropyFormulationForCensoredLevyOperator} to get
 \begin{equation*}
    \begin{split}
        &\,-\int_0^1\int_Q (u-k)^\pm_\e\partial_t\varphi + F^\pm_\e(u,k)\nabla\varphi\dd x \dd t\dd a\\
        &\,- \int_0^1\int_Q \sgn^{\pm}_\e(u-k) \L^{\geq r}[b(u)]\varphi \dd x\dd t\dd a - \int_0^1\int_M(b(u)-b(k))_\e^\pm\L^{<r}[\varphi]\dd x\dd t\dd a\\
    \leq &\, \int_{\Omega}(u_0-k)^{\pm}_\e\varphi(0,\cdot)\dd x + L_f\int_{\Gamma}(\overline{u}^c - k)^{\pm}_\e\varphi \dd \sigma(x)\dd t.
    \end{split}
\end{equation*}
Finally, we send $\e\to0$ to get inequality \eqref{eq: entropyProcessInequalityVovelleMethod}, using the dominated convergence theorem  and the pointwise convergence of $g_\e^\pm\to g^\pm$ for $g^\pm=(a-\cdot)^\pm$, $\sgn^{\pm}(a-\cdot)$, and $F^\pm(a,\cdot)$. Thus, $u$ is an entropy-process solution of \eqref{E} which, as noted at the beginning of the proof, implies the proposition.
\end{proof}
  \begin{remark}\label{rem: whyWeNeedStrongConvergenceOfBUN}
       The reason strong convergence of $(b(u_n))_{n\in \N}\subset L^2(Q)$ is needed in the previous proof, is to ensure  that in the limit we can replace $u_n$ by $u$ in        
       the nonsingular nonlocal term
\begin{equation}\label{eq: theProblemWithTheNonsingularPartOfTheDefinition}
   \int_Q \big(b(u_n)\ast \mu_r\big)\sgn^\pm(u_n-k)\varphi\dd x\dd t -\mu_r(\R^d)\int_Q b(u_n)\sgn^\pm(u_n-k)\varphi\dd x\dd t,
\end{equation}
here written in the style of Remark \ref{rem: defentsolnremark} (b). Ignoring that $\sgn^{\pm}$ is discontinuous (cf. the previous proof), 
the nonlinear weak-$\star$ convergence of $u_n$ assigns a value to the second part of \eqref{eq: theProblemWithTheNonsingularPartOfTheDefinition}, but \textit{not} to the first part which  (by Fubini) contains terms of the type $b(u_n(x-z))\sgn^{\pm}(u_n(x)-k)$. 
To illustrate the problem, consider $g_n(\theta)=\sin(n\theta)$ which, as $n\to\infty$, admits the nonlinear weak-$\star$ limit $g(\theta,a)=2a-1$, yet the product $g_n(\theta)g_n(\theta + \pi) = (-1)^n\sin(n\pi)^2$ admits no weak limit. 
  \end{remark}


\section{From energy estimates to translation regularity
}\label{sec: precompactnessOfB(u_n)}

As in the previous section, we let $\L_n$ be an approximation of $\L$ satisfying \eqref{muassumptionlim} and $u_n$ the corresponding entropy solution of the Dirichlet problem \eqref{E2} (data is independent of $n$). We will investigate what conditions on $(\L_n)_{n\in\N}$ ensures precompactness of the diffusive fluxes $(b(u_n))_{n\in\N}$ which is needed to apply Proposition \ref{prop: stabilityOfEntropySolutionsWithRespectToPerturbationsInL}. For technical convenience, we extend these functions to $\R^{1+d}$: More precisely, we let $(\gamma_n)_{n\in\N}\in L^2(\R^{1+d})$ be defined by
\begin{equation}\label{eq:DefinitionExtensionB}
\gamma_n(t,x)\coloneqq
\begin{cases}
b(u_n(t,x))-b(\overline{u}^c(t,x))\qquad&\text{a.e. in} \quad (0,T)\times\R^d,\\
0, \qquad &\text{a.e. in} \quad (0,T)^c\times\R^d.
\end{cases}
\end{equation}
 Note that $(b(u_n))_{n\in\N}$ is precompact in $L^2(Q)$ if, and only if, $(\gamma_n)_{n\in\N}$ is precompact in $L^2(\R^{1+d})$. And since $(\gamma_n)_{n\in\N}$ is uniformly $L^\infty$-bounded (by Lemma \ref{lem: maximumsPrinciple}) and each $\gamma_n$ is supported in $Q$, it follows from Kolmogorov's compactness theorem that $(\gamma_n)_{n\in\N}$ is precompact if it is $L^2$-equicontinuous with respect to translations. Furthermore, it suffices that $(\gamma_n)_{n\in\N}$ is equicontinuous with respect to \textit{spatial}-translations: 

\begin{proposition}
\label{prop: fromSpatialTranslationContinuityToTemporalTranslationContinuity}
Assume \eqref{Omegaassumption}--\eqref{muassumption},  \eqref{muassumptionlim}, $u_n$ is the entropy solution of \eqref{E2}, 
and $\gamma_n$ is defined in \eqref{eq:DefinitionExtensionB}. If $(\gamma_n)_n$ is $L^2$-equicontinuous with respect to space-translations
\begin{equation}\label{eq:SpaceTranslations}
\sup_n\|\gamma_n(\cdot,\cdot+h)-\gamma_n(\cdot,\cdot)\|_{L^2(\R^{1+d})}\to 0\qquad\text{as}\qquad h\to0,
\end{equation}
then it is also $L^2$-equicontinuous with respect to time-translations
\begin{equation}\label{eq: timeTranslationContinuity}
\sup_n\|\gamma_n(\cdot+\tau,\cdot)-\gamma_n(\cdot,\cdot)\|_{L^2(\R^{1+d})}\to 0\qquad\text{as}\qquad \tau\to0,
\end{equation}
and, consequently, $(b(u_n))_{n\in\N}$ is precompact in $L^2(Q)$.
\end{proposition}
 This is relatively classical type of result that can be 
proved through a mollification argument.  Let  
$\rho_\varepsilon(x)=\frac1{\e^d}\rho(\frac x\e)$, where $0\leq \rho\in C_c^\infty(\R^d)$ is supported in $B_1(0)$ and integrates to one,  and set 
\begin{equation*}
u_{n,\varepsilon}(t,x)\coloneqq(u_{n}(t,\cdot)\ast\rho_\varepsilon)(x)= \int_{\R^d}u_{n}(t,y)\rho_\varepsilon(x-y)\dd y.
\end{equation*} 
 Uniform (in $n\in\N$) control on $\partial_t u_{n,\varepsilon}$  implies uniform time-translation regularity of $b(u_n)$ if we have  uniform control on the difference $b(u_{n,\varepsilon})-b(u_{n})$.  Spatial translation estimates for $b(u_n)$ 
provides control on the difference $b(u_n)\ast \rho_\varepsilon - b(u_n)$; this is a standard argument. But the difference of interest here is instead $b(u_n\ast \rho_\varepsilon) - b(u_n)$ and, thus, it is not immediately clear how controlling the translates in $b(u_n)$ can help us. This is resolved by the following lemma.

\begin{lemma}\label{lem:e-t-reg}
Assume \eqref{Omegaassumption}--\eqref{muassumption}, \eqref{muassumptionlim}, $\e>0$, $u_n$ is the entropy solution of \eqref{E2}, and $\rho_\varepsilon$ is the mollifier from before. Then we have the pointwise bound
\begin{equation}\label{eq:SquarePointwiseControl}
\begin{split}
\big|b\big(u_{n}(t,\cdot)\ast\rho_\varepsilon(x)\big)-b\big(u_n(t,x)\big)\big|^2
\leq &\, C\big|b\big(u_n(t,\cdot)\big)-b\big(u_n(t,x)\big)\big|\ast\rho_\varepsilon(x),
\end{split}
\end{equation}
for a.e. $t\in(0,T)$ and a.e. $x\in\Omega_{-\e}$, where $C=2L_b(\norm{u_0}{L^\infty(\Omega)}\vee\norm{u^c}{L^\infty(Q^c)})$.
\end{lemma}
\begin{remark}
    When  $b'\geq c>0$ we have a much simpler bound,
    \begin{align*}
        \big|b\big(u_{n}\ast\rho_\varepsilon\big)-b\big(u_n\big)\big|\leq L_b |u_n(t,\cdot) - u_n|\ast\rho_\varepsilon\leq  cL_b  |b(u_n(t,\cdot)) - b(u_n)|\ast\rho_\varepsilon.
    \end{align*}
 The  above lemma 
 allows  us to \textit{always} move $\ast\rho_\varepsilon$ out of $b$ at the 
cost of introducing a square root.
\end{remark}

\begin{proof} We will use Lemma \ref{lem: theAverageAlmostCommuteWithVShapedFunction}. Fix $n\in\N$ and $(t,x)\in (0,T)\times\Omega_{-\varepsilon}$. Define $w(y)\coloneqq u_n(t,y)-u_n(t,x)$. Note that $|w|\leq R\coloneqq  2\big(\norm{u_0}{L^\infty(\Omega)}\vee\norm{u^c}{L^\infty(Q^c)}\big)$ by the $L^\infty$-bound Lemma \ref{lem: maximumsPrinciple}. Define further the probability measure $\kappa$ on $[-R,R]$ as the $w$-pushforward measure
\begin{align*}
    \kappa (I) \coloneqq \int_{w^{-1}(I)}\rho_{\e}(x-y)\dd y\qquad\text{for Borel}\quad I\subseteq [-R,R].
\end{align*}
For any continuous function $g\colon [-R,R]\to \R$, we then get
\begin{align}\label{eq: thePushForwardIdentity}
    \int_{[-R,R]} g(s)\dd \k(s) = \int_{\R^d}g(w(y))\rho_\e(x-y)\dd y = \int_{\R^d}g(u_n(t,y)-u_n(t,x))\rho_\e(x-y)\dd y.
\end{align}
Letting $h\colon [-R,R]\to[0,\infty)$ be defined by
\begin{align*}
    h(s)\coloneqq |b(s + u_n(t,x))-b(u_n(t,x))|,
\end{align*}
we infer from \eqref{eq: thePushForwardIdentity} that
\begin{align*}
    \overline{s}\coloneqq &\, \int_{[-R,R]}s\dd\k(s) = u_{n}(t,\cdot)\ast \rho_\varepsilon (x) - u(t,x),\\
     \overline{h}\coloneqq&\, \int_{[-R,R]}h(s)\dd\k(s) =|b(u_n(t,\cdot))-b(u_n(t,x))|\ast\rho_\e(x).
\end{align*}
By Lemma \ref{lem: theAverageAlmostCommuteWithVShapedFunction} it follows that $h(\ol s)^2\leq L_bR\ol h$ which is precisely \eqref{eq:SquarePointwiseControl}.
\end{proof}

\begin{proof}[Proof of Proposition \ref{prop: fromSpatialTranslationContinuityToTemporalTranslationContinuity}]\
\medskip
\noindent \textbf{1)} \ Lemma A.1 (b) in \cite{AlEnJaMa23} holds for entropy solutions of \eqref{E2}, i.e., entropy solutions are very weak solutions: for all $\varphi\in C_c^\infty(Q)$,
\begin{equation}\label{eq:EntropySolutionsAreWeakSolutions}
-\int_Q\Big(u_n\partial_t\varphi+f(u_n)\cdot \nabla\varphi\Big)\dd x \dd t+\int_Mb(u_n)\L_n[\varphi]\dd x \dd t=0.
\end{equation}

\smallskip
\noindent \textbf{2)} \ Now we estimate time translations of $u_{n,\e}$ via the estimate for $\partial_tu_n$ given by the weak formulation. 
Roughly speaking, from \eqref{eq:EntropySolutionsAreWeakSolutions} with $y$ instead of $x$ and $\varphi(t,y)\coloneqq \chi_{(t_1,t_2)}(t)\rho_\varepsilon(x-y)$ for $t_1,t_2\in(0,T)$, we find that 
\begin{equation*}
u_{n,\e}(t_2,x)-u_{n,\e}(t_1,x)=\int_{t_1}^{t_2}\Big( f(u_n)*_x\nabla\rho_\varepsilon(x)-b(u_n)*_x\L_n[\rho_\varepsilon](x)\Big)\dd t\quad\text{for}\quad x\in \Omega_{-\e},
\end{equation*}
where $f(u_n)\ast \nabla\rho_\varepsilon\coloneqq \sum_{j=1}^df_j(u_n)\ast \partial_j\rho_\varepsilon$. To make this argument rigorous, we must approximate $\chi_{(t_1,t_2)}$ and then pass to the limit. The result will then hold in Lebesgue points $t_1,t_2$ of the time slices of $u_{\e}$ and hence for a.e. $t_1,t_2\in(0,T)$.\footnote{Cf. e.g. \cite[Lemma B.1]{Ali07} for more details.}
Integrating over $\Omega_{-\e}$ we get
\begin{equation*}
\int_{\Omega_{-\varepsilon}}\!|u_{n,\varepsilon}(t_2,x)-u_{n,\varepsilon}(t_1,x)|\dd x
\leq |t_2-t_1||\Omega|\Big(\norm{f(u_n)}{L^\infty(Q)}\norm{\nabla \rho_\e}{L^1(\R^d)}+\norm{b(u_n)}{L^\infty(M)}\norm{\L[\rho_\varepsilon]}{L^1(\R^d)}
\Big).
\end{equation*}
Since by standard estimates, 
$\norm{\nabla\rho_\varepsilon}{L^1(\R^d)}=\varepsilon^{-1}\norm{\nabla\rho}{L^1(\R^d)}$ and
\begin{equation*}
\norm{\L[\rho_\varepsilon]}{L^1(\R^d)}\leq \frac12\norm{D^2\rho}{L^1(\R^d)}\frac1{\e^2}\int_{|z|<1}|z|^2\dd\mu_n(z) + 2\|\rho\|_{L^1(\R^d)}\int_{|z|\geq1}\dd\mu(z),
\end{equation*}
this means that for a.e.  $t,t+\tau\in(0,T)$,
\begin{equation}
\label{eq:TimeTranslationForMollifiedSolutionN}
\sup_{n\in\N}\int_{\Omega_{-\varepsilon}}|u_{n,\varepsilon}(t+\tau,x)-u_{n,\varepsilon}(t,x)|\dd x\leq C|\tau|\big(1+\varepsilon^{-1}+\varepsilon^{-2}\big),
\end{equation}
where
$C:= \sup_{n\in\N}2|\Omega|\int_{\R^d}\big(|z|^2\wedge1\big)\dd \mu(z)\norm{\rho}{W^{2,1}}(\norm{f(u_n)}{L^\infty}+\norm{b(u_n)}{L^\infty}$ $+\norm{b(u^c)}{L^\infty})<\infty$ by assumptions and Lemma \ref{lem: maximumsPrinciple}.

\smallskip
\noindent \textbf{3)} \ We now prove the $L^2$ time translation result \eqref{eq: timeTranslationContinuity} for $\gamma_n$. By symmetry, it suffices to consider $0<\tau<T$. 
Divide the domain into $Q_{ -\tau, -\varepsilon}\coloneqq (0,T-\tau)\times \Omega_{-\varepsilon}$ and $Q_{-\tau,-\varepsilon}^c=(\R\times \R^d)\setminus Q_{-\tau,-\varepsilon}$, and for effective notation we define the time-shift operator $S_\tau\colon  g(t)\mapsto g(t+\tau)$. We first split $\|\gamma(\cdot+\tau,\cdot)-\gamma_n\|_{L^2(\R^{1+d})}=\|(S_\tau-1)\gamma_n\|_{L^2(\R^{1+d})}$ into
\begin{equation}\label{eq: splittingTheGammaDifferencesInTwo}
\begin{split}
&\|(S_\tau-1)\gamma_n\|_{L^2(\R^{1+d})}\leq  \|(S_\tau-1)\gamma_n\|_{L^2(Q_{-\tau,-\varepsilon})}+\|(S_\tau-1)\gamma_n\|_{L^2(Q_{-\tau,-\varepsilon}^c)}.
\end{split}   
\end{equation}
The part on $Q^c_{-\tau,-\varepsilon}$ is easily bounded by using $\supp \gamma_n \subseteq Q$, \eqref{Omegaassumption}, \eqref{u^cassumption}, and Lemma \ref{lem: maximumsPrinciple},
\begin{align*}
    \|(S_\tau-1)\gamma_n\|_{L^2(Q_{-\tau,-\varepsilon}^c)}\leq&\,\|S_\tau\gamma_n\|_{L^2(Q_{-\tau,-\varepsilon}^c)}+\|\gamma_n\|_{L^2(Q_{-\tau,-\varepsilon}^c)}
    \leq 2\|\gamma_n\|_{L^\infty(Q)}|Q\setminus Q_{-\tau,-\varepsilon}|\lesssim \tau + \varepsilon,
\end{align*}
for some implicit constant independent of $n$. As for the part on $Q_{-\tau,-\varepsilon}$, we first estimate
\begin{align*}
&\|(S_\tau-1)\gamma_n\|_{L^2(Q_{-\tau,-\varepsilon})} \lesssim \|(S_\tau-1)b(u_n)\|_{L^2(Q_{-\tau,-\varepsilon})}+\tau,
\end{align*}
by the regularity of $\bu$. We then add/subtract $(S_\tau-1) b(u_{n,\varepsilon})$ to the remaining term and get
\begin{align*}
\|&\,(S_\tau-1)b(u_n)\|_{L^2(Q_{-\tau,-\varepsilon})}\leq \|(S_\tau-1) b(u_{n,\varepsilon})\|_{L^2(Q_{-\tau,-\varepsilon})} + 2\|b(u_{n,\varepsilon})-b(u_{n})\|_{L^2(Q_{0,-\varepsilon})}.
\end{align*}
These two latter terms can be controlled: Using \eqref{eq:TimeTranslationForMollifiedSolutionN} we have 
\begin{align*}
\|(S_\tau-1) b(u_{n,\varepsilon})\|_{L^2(Q_{-\tau,-\varepsilon})}\leq&\,L_b\|u_{n}\|^{\frac12}_{L^\infty(Q)}\|(S_\tau-1)u_{n,\varepsilon}\|^{\frac12}_{L^1(Q_{-\tau,-\varepsilon})} \lesssim \tau^{\frac12}(1+\varepsilon^{-1}) ,
\end{align*}
and by Lemma \ref{lem:e-t-reg} and the assumption \eqref{eq:SpaceTranslations} we also have
\begin{align*}
\|b(u_{n,\varepsilon})-b(u_{n})\|_{L^2(Q_{0,-\varepsilon})}\leq&\,2C\Big(\int_{Q_{0,-\varepsilon}}\int_{\R^d}|b(u_n(t,y))-b(u_n(t,x))|\rho_\varepsilon(x-y)\dd y\dd x\dd t\Big)^{\frac{1}{2}}\\
\lesssim &\,\sup_{|h|\leq \varepsilon}\big\|b\big(u_n(\cdot,\cdot+h)\big)-b\big(u_n\big)\big\|_{L^1(Q_{0,-\varepsilon})}^{\frac{1}{2}}\\
\leq&\, |Q|^{\frac{1}{4}}\sup_{|h|\leq \varepsilon}\big\|b\big(u_n(\cdot,\cdot+h)\big)-b\big(u_n\big)\big\|_{L^2(Q_{0,-\varepsilon})}^{\frac{1}{2}}\\
\lesssim &\, \sup_{|h|\leq \varepsilon}\|\gamma_n(\cdot,\cdot+h)-\gamma_n\|_{L^2(Q_{0,-\varepsilon})}^{\frac{1}{2}} + \varepsilon^{\frac{1}{2}} = o_{\varepsilon}(1),
\end{align*}
for an implicit constant independent of $n$. Substituting all these estimates into \eqref{eq: splittingTheGammaDifferencesInTwo} we finally obtain
\begin{align*}
    \sup_{n}\|(S_\tau-1)\gamma_n\|_{L^2(\R^{1+d})} \lesssim \tau^{\frac{1}{2}}\varepsilon^{-1} + o_\tau(1) + o_\varepsilon(1).
\end{align*}
Setting $\varepsilon =\tau^{\frac{1}{3}}$ we obtain \eqref{eq: timeTranslationContinuity}. The last part of the proposition is evident in light of the discussion at the beginning of the section. 
\end{proof}

 In light of the previous proposition, we are lead to examine what equicontinuity with respect to spatial-translations we can derive for the family $(\gamma_n)_{n\in\N}$. Had 
\eqref{E} been  posed in  all of $\R^d$, then $L^1$-contractions results would 
 yield  such equicontinuity. Lacking this option, we instead seek to exploit the energy estimate, Proposition \ref{prop: finiteEnergy}. A first step in this direction is the following observation. 
\begin{lemma}[Uniform energy estimate]\label{lem: uniformEnergy}
Assume \eqref{Omegaassumption}--\eqref{muassumption},  \eqref{muassumptionlim}, $u_n$ is the entropy solution of \eqref{E2}, 
and $\gamma_n$ is defined in \eqref{eq:DefinitionExtensionB}.
Then the sequence $(\gamma_n)_{n\in\N}$ has uniformly bounded energy, that is  
\begin{align}\label{energy_n}
    \sup_n\int_0^T\int_{\R^d}\B_n[\gamma_n,\gamma_n]\dd x \dd t <\infty,
\end{align}
where $\B_n$ is the bilinear operator corresponding to $\L_n$, cf. \eqref{def:blform}.
\end{lemma}
\begin{proof}
Recall that $\gamma_n(t,x)=b(u_n(t,x))-b(\ol{u}^c(t,x))$ in $M=(0,T)\times\R^d$. By the above assumptions and Proposition \ref{prop: finiteEnergy}, the bound \eqref{eq: finiteEnergy} holds for the solution $u_n$ of \eqref{E2}, with a right-hand side that is bounded uniformly in $n$. By the assumptions this is immediate for the two first terms, and now we check the third term: 
\begin{align*}
    \Big|\int_{Q} \L_n[b(\bu)](b(u_n)-b(\bu_n))\dd x\dd t\Big| \leq \|\L_n[b(\bu)]\|_{L^1(Q)}(\|b(u_n)\|_{L^\infty}+\|b(\bu)\|_{L^\infty}).
\end{align*}
The first factor is bounded in $n$ by \cite[Corollary 3.7]{AlEnJaMa23} and the second by Lemma \ref{lem: maximumsPrinciple}. 
\end{proof}
 Looking at the definition of $\B$, it is natural to expect that an 
energy estimate like Lemma \ref{lem: uniformEnergy} implies  continuity of 
space translations  
 in $L^2$  of 
$\gamma_n$. The problem 
 here  is to find some 
condition on $(\L_n)_{n\in\N}$  under  which this continuity is uniform in $n$.  
 It is  easier to study this problem on the Fourier-side, where Lévy operators become multipliers 
 and  the condition can be characterized in a sharp manner  -- see  
Proposition \ref{prop: uniformTranslationContinuityForASequenceOfFiniteEnergies} below. 
 The (Fourier) multiplier ${-m}$ of $\L$ is given by\footnote{ This follows from the Lévy-Kinchine theorem \cite{App09} or a simple direct computation, see Lemma \ref{lem: fourierCharacterizationsOfLAndB}; in particular, $-m$ is the characteristic exponent in the Lévy–Khintchine representation of the Lévy process generated by $\L$ \cite{App09}.}
\begin{align}\label{eq: definitionOfSymbolOfLevyOperator}
    m(\xi)\coloneqq \int_{\R^d}\Big( 1 - \cos(\xi\cdot z)\Big)\dd\mu(z)\qquad\text{for}\qquad \xi\in\R^d,
\end{align}
where the integral is well-defined by \eqref{muassumption} since the integrand is bounded by $(\frac{1}{2}|\xi|^2|z|^2)\wedge 2$.
The function $m$ (also called multiplier from time to time) is continuous, nonnegative, and symmetric about zero. (It is also negative definite in Schoenberg sense \cite{BeFo75}, but this will not matter here.) 
 
 By the Frechet-Kolmogorov theorem, precompact sequences in $L^2$ are equicontinuous with respect to translations. The next result gives a necessary and sufficient condition in terms of $m$ for the energy space $H_0^{\L}(\Omega)$ defined through $\B$ to be compactly embedded in $L^2$. This result is of independent interest, and though it may be classical, we have not found it in the literature. 

\begin{proposition}\label{prop: coerciveSymbolEqualsCompactEmbeddingOfLevySpaceInL2}
Let $\L$ be a symmetric Lévy operator, $m$ its multiplier, and let $H_0^{\L}(\Omega)$ be as in \eqref{eq: definitionOfLevySpaceNorm}.  Then the embedding $H^{\L}_0(\Omega)\hookrightarrow L^2(\Omega)$ is compact if, and only if, 
\begin{align}\label{eq: criterionOnMeasureForCompactness}
\liminf_{|\xi|\to\infty}m(\xi)=\infty,
\end{align}
 which is precisely assumption \eqref{muassumption2}.
\end{proposition}
 A proof and discussion is given in Appendix \ref{sec:LevyOperatorsInducingTranslationRegularity}.

\begin{remark} 
In general, there is no simple way of ‘reading off’ from $\mu$ whether the corresponding multiplier $m$ satisfies condition \eqref{eq: criterionOnMeasureForCompactness}. Still, it is necessary, but not sufficient, that $\mu(\R^d)=\infty$, and it is sufficient, but not necessary, that $\mu_a(\R^d)=\infty$ where $\mu_a$ is the absolutely continuous part of $\mu$. All this follows from Lemmas \ref{lem: zeroOrderLevyOperatorsHaveBoundedSymbols}, \ref{lem: absolutelyContinuousMeasureHasNiceSymbol}, and \ref{lem: examplesOfBadLevyOperators}. 
\end{remark}

While this result is not directly applicable in our situation -- the sequence $(\gamma_n)_{n\in\N}$ is not bounded in a common $H^{\L}_0$-space -- it, and its proof, can be modified without much difficulty to suit our needs:
\begin{proposition}[Spatial equicontinuity from finite energy]\label{prop: uniformTranslationContinuityForASequenceOfFiniteEnergies}
    Consider a bounded sequence of functions $(\phi_n)_{n\in\N}\subset L^2(\R\times \R^d)$ and a sequence $(\L_n)_{n\in\N}$ of symmetric Lévy operators satisfying \eqref{muassumption} with corresponding bilinear operators $(\B_n)_{n\in \N}$ and multipliers $(m_n)_{n\in \N}$. If we have the uniform energy estimate 
    \begin{equation}\label{eq: theGeneralEnergyCondition}
        \sup_{n}\int_\R\int_{\R^d}\B_n[\phi_n(t,\cdot),\phi_n(t,\cdot)](x)\dd x\dd t <\infty,
    \end{equation}
and the joint coercivity condition on the multipliers
\begin{equation}\label{eq: theGeneralCoercivityConditionOfSymbol}
    \liminf_{\substack{|\xi|,n\to\infty}}m_n(\xi) = \infty,
\end{equation}
then the family $(\phi_n)_{n\in\N}$ is $L^2$-equicontinuous with respect to translations in space, that is
\begin{equation*}
    \lim_{y\to0}\sup_{n}\int_\R\int_{\R^d}|\phi_n(t,x+y)-\phi_n(t,x)|^2\dd x\dd t =0.
\end{equation*}
\end{proposition}
\begin{remark}\label{remark: aSimplificationToTheComplicatedCorecivityCondition}  When the multipliers are pointwise decreasing $m_1\geq m_2\geq\dots$, the coercivity condition \eqref{eq: theGeneralCoercivityConditionOfSymbol} reduces to 
\begin{equation*}
    \liminf_{|\xi|\to\infty}\lim_{n\to\infty} m_n(\xi)=\infty,
\end{equation*}
and if they are pointwise increasing $m_1\leq m_2\leq\dots$,  the condition reduces to 
\begin{equation*}
\lim_{n\to\infty}\liminf_{|\xi|\to\infty}m_n(\xi)=\infty.
\end{equation*}
In particular, the multipliers are pointwise increasing when $\L_n=\L^{\geq 1/n}$ for some fixed Lévy operator $\L$; this is because the integrand of \eqref{eq: definitionOfSymbolOfLevyOperator} is non-negative and $\{|z|\geq \tfrac{1}{n}\}\subset \{|z|\geq \tfrac{1}{n+1}\}$.
\end{remark}
\begin{proof}[Proof of Proposition \ref{prop: uniformTranslationContinuityForASequenceOfFiniteEnergies}]
    Define $\rho\colon[0,\infty)\to[0,\infty)$ by
\begin{align*}
    \rho(R)\coloneqq \inf_{|\xi|,n\geq  R}m_n(\xi),
\end{align*}
which, in light of \eqref{eq: theGeneralCoercivityConditionOfSymbol}, tends to infinity as $R\to\infty$; in particular, we can fix a sufficiently large $R>0$ such that $\rho(R)>0$. Next, for arbitrary $n>R$ and $y\in\R^d$ we use Plancherel's theorem and compute
\begin{align*}
    &\,\norm{\phi_n(\cdot,\cdot + y)-\phi_n}{L^2(\R\times \R^{d})}^2 = \frac{1}{(2\pi)^d}\int_\R\int_{\R^d}|1-e^{i\xi \cdot y}|^2|\hat{\phi}_n(t,\xi)|^2\dd\xi\dd t\\
    \leq &\,\frac{(R|y|)^2}{(2\pi)^d}\int_\R\int_{|\xi|<R}|\hat{\phi}_n(t,\xi)|^2\dd\xi\dd t + \frac{4}{\rho(R)(2\pi)^d}\int_\R\int_{|\xi|\geq R}m_n(\xi)|\hat \phi_n(t,\xi)|^2\dd\xi\dd t\\ 
    \leq&\, (R|y|)^2 c_1 + \frac{4c_2}{\rho(R)}.
\end{align*}
Here we defined $ c_1\coloneqq \sup_n \norm{\phi_n}{L^2(\R\times\R^d)}^2<\infty$ and
\begin{align*}
     c_2\coloneqq \sup_n \frac{1}{(2\pi)^d}\int_\R\int_{\R^d}m_n(\xi)|\hat \phi_n(t,\xi)|^2\dd\xi\dd t=\sup_n \int_\R\int_{\R^d}\B_n[\phi_n(t,\cdot),\phi_n(t,\cdot)](x)\dd\xi\dd t<\infty,
\end{align*}
where the middle equality is the identity \eqref{eq: theLevyHilbertNormOnTheForuierSide} from Lemma \ref{lem: fourierCharacterizationsOfLAndB}. We conclude that  
\begin{align*}
   &\, \lim_{y\to 0}\sup_{n} \norm{\phi_n(\cdot,\cdot + y)-\phi_n}{L^2(\R\times \R^{d})}^2\\ 
    \leq&\, \lim_{y\to 0}\Bigg[ \Big(\max_{n<R} \norm{\phi_n(\cdot,\cdot + y)-\phi_n}{L^2(\R\times \R^{d})}^2\Big)\vee\bigg((R|y|)^2 c_1 + \frac{4c_2}{\rho(R)}\bigg)\Bigg]
    = \frac{4c_2}{\rho(R)},
\end{align*}
where the last equality follows as a finite number of $L^2$ functions are uniformly continuous with respect to translations. Sending $R\to\infty$, the result follows.
\end{proof}
We end the section with a remark on the subtle fact that the multipliers $m_n$ of the truncated Lévy operators $\L^{\geq 1/n}$ may fail \eqref{eq: theGeneralCoercivityConditionOfSymbol} even when $m$ of $\L$ satisfies \eqref{eq: criterionOnMeasureForCompactness}:
\begin{remark}\label{rem: thereExistProblematicLevyoperators}
    The most direct application of the previous proposition is to the case $\L_n=\L^{\geq 1/n}$ where $\L$ is some fixed Lévy operator. However, as explained in Lemma \ref{lem: examplesOfBadLevyOperators}, the corresponding sequence of multipliers $(m_n)_{n\in\N}$ may fail the coercivity condition \eqref{eq: theGeneralCoercivityConditionOfSymbol} even though $\L$ satisfies \eqref{muassumption2}. Such `bad' Lévy operators complicates the existence proof  since one cannot necessarily approximate $\L$ in a sufficiently strong sense through its (bounded) truncations. Conveniently, this is by Lemma \ref{lem: absolutelyContinuousMeasureHasNiceSymbol} not a problem for the fractional Laplacians, and so the general existence proof -- for when $\L$ satisfies \eqref{muassumption2} -- bootstraps off of this simpler example which \textit{can} be sufficiently approximated by truncations.
\end{remark}

We conclude the section with the following summary of the derived results, obtained by combining Propositions \ref{prop: fromSpatialTranslationContinuityToTemporalTranslationContinuity} and \ref{prop: uniformTranslationContinuityForASequenceOfFiniteEnergies} and Lemma \ref{lem: uniformEnergy}.
\begin{corollary}[Precompactness of $(b(u_n))_{n\in \N}$ from finite energies]\label{cor: jointCorecivityImpliesPrecompactnessOfb(u_n)}
Assume \eqref{Omegaassumption}--\eqref{muassumption},  \eqref{muassumptionlim}, $u_n$ is the entropy solution of \eqref{E2}, and that the corresponding multipliers $(m_n)_{n\in \N}$ of $(\L_n)_{n\in\N}$ satisfy \eqref{eq: theGeneralCoercivityConditionOfSymbol}, i.e.  
$\liminf_{|\xi|,n\to \infty}m_n(\xi)=\infty$. 
Then $(b(u_n))_{n\in \N}\subset L^2(Q)$ is precompact. 
\end{corollary}

\section{Existence of entropy solutions} 
\label{sec: existenceOfEntropySolutions}

 In this section we give the proofs of our existence results. The arguments rely on the compactness and stability results obtained in Section \ref{sec: stabiltyWithRespectToVariationsInL}. 
Note that by Lemma \ref{lem: maximumsPrinciple}, the solutions we seek will take values in the  interval
$$
\Big[\essinf_{\Omega} u_0\wedge \essinf_{Q^c} u^c, \esssup_{\Omega} u_0\vee  \esssup_{ Q^c} u^c\Big].
$$ 
Modifications of $f,b$ outside this interval will have no effect on the solutions of the problem. Without loss of generality, we may therefore assume $f$ and $b$ to be globally Lipschitz continuous and bounded. 
With this in mind we start the existence proof from the following local problem:
\begin{align}\label{eq: conservationLawWithSourceTerm}
    \begin{cases}
    \partial_tu +\text{div}\big(f(u)\big)=g,\qquad&\text{in}\qquad Q,\\
u=\overline{u}^c \qquad&\text{on}\qquad \Gamma,\\
u(0,\cdot)=u_0 \qquad&\text{on}\qquad \Omega.\\
    \end{cases}
\end{align}

\begin{lemma}[\cite{PoVo03}]\label{lem:CL}
Assume \eqref{Omegaassumption}, \eqref{fassumption}, \eqref{u^cassumption}, \eqref{u_0assumption}, $g\in L^1(Q)$, and $f$ globally Lipschitz and bounded. Then
\begin{enumerate}[{\rm (a)}]
\item There exists a unique renormalised
entropy solution $u\in C(0,T;L^1(\Omega))$ of \eqref{eq: conservationLawWithSourceTerm}.
\item Two renormalised entropy solutions $u,\tilde {u}\in L^1(Q)$ of \eqref{eq: conservationLawWithSourceTerm} with source terms $g,\tilde g\in L^1(Q)$, satisfy \begin{align}\label{eq: L1contractionOfRenormalizedEntropySolutions}
    \int_{\Omega}|u-\tilde u|(t) \dd x\leq 
    \int_0^t\int_{\Omega}|g-\tilde g|\dd x \dd t \quad \text{for}\quad t\in (0,T).
\end{align}
\item If in addition $g\in L^\infty(Q)$, then $u\in L^\infty(Q)$, and $u$ is an entropy solution of \eqref{eq: conservationLawWithSourceTerm}: For all $k\in\R$ and $0\leq \varphi\in C^{\infty}_c([0,T)\times\R^d)$, \begin{equation}\label{eq: entropyFormulationBasicConservationLaw}
    \begin{split}
        &\,-\int_Q (u-k)^\pm\varphi_t + F^\pm(u,k)\cdot\nabla\varphi + \mathrm{sgn}^{\pm}(u-k)g\varphi\dd x\dd t\\
    \leq &\,  \int_{\Omega}(u_0-k)^{\pm}\varphi(0,\cdot)\dd x + L_f\int_{\Gamma}(\overline{u}^c - k)^{\pm}\varphi \dd \sigma(x)\dd t.
    \end{split}
\end{equation}
\end{enumerate}
\end{lemma}

\begin{proof}
Since $\Omega, u, u_0, f$ and $g$ satisfy the assumptions of \cite{PoVo03}, existence and uniqueness of renormalised entropy solutions of \eqref{eq: conservationLawWithSourceTerm} follow from \cite[Theorem 2.3]{PoVo03} and the $L^1$-contraction from \cite[Theorem 3.1]{PoVo03}. Under our assumptions (data in $L^\infty$) the renormalized entropy solutions of \cite{PoVo03} are entropy solutions and hence satisfy \eqref{eq: entropyFormulationBasicConservationLaw}, see \cite[Remark 2.2]{PoVo03}. The solution $u$ is bounded by (adapting the proof of) Lemma \ref{lem: maximumsPrinciple}. 
\end{proof}

Combining this result with a fixed point argument, we obtain the existence of solutions for \eqref{E} when $\L$ is bounded ($\mu$ is finite).

\begin{theorem}[Existence for bounded $\L$]\label{thm: existenceForBoundedL}
Assume \eqref{Omegaassumption}--\eqref{u_0assumption}, $f, b$ globally Lipschitz and bounded, and 
$\mu\geq0$ is a symmetric Radon measure such that 
$\mu(\R^d)<\infty$. Then there exists an entropy solution $u$ of \eqref{E}.
\end{theorem}

\begin{proof}
On $C(0,T;L^1(\Omega))$ we define the mapping $I\colon v\mapsto u$, where $u$ is the renormalised entropy solution of \eqref{eq: conservationLawWithSourceTerm} with $g=\L[b(\overline{v})]$ and where $\overline{v}$ is the extension of $v$ to $M$ set to satisfy $\overline{v}_{|Q^c}=u^c$. Since $\overline{v}$ is locally integrable, $b$ is bounded, and $\mu$ is finite, it follows that $g_{|Q}\in L^\infty(Q)\subset L^1(Q)$. Since $C(0,T;L^1(\Omega))\subset L^1(Q)$, it follows by Lemma \ref{lem:CL} that  $I\colon C(0,T;L^1(\Omega))\to C(0,T;L^1(\Omega))$ is well-defined, injective, and Lipschitz-continuous
\begin{equation}\label{eq: sliceContinuityOfTheMapI}   
        \norm{(I[v]-I[w])(t)}{L^1(\Omega)}\leq  \int_0^t\norm{\L[b(\overline{v})-b(\overline{w})](s)}{L^1(\Omega)} \dd s\leq 2L_b\|\mu\|\int_0^t\norm{(v-w)(s)}{L^1(\Omega)}\dd s,
 \end{equation}
 for 
 $t\in (0,T)$ where $L_b$ is the Lipschitz constant of $b$ and $\|\mu\|$ the total mass of $\mu$.
 
Consider the fixed point iteration $u_{k+1}=I[u_k]$ for $k\in\N_0$ and $u_0=0$. From \eqref{eq: sliceContinuityOfTheMapI} we get 
\begin{align*}
     \norm{(u_{k+1}-u_k)(t)}{L^1(\Omega)} \leq \ &2L_b\|\mu\|\int_0^t\norm{(u_k-u_{k-1})(s)}{L^1(\Omega)}\dd s\\
     \leq \ &(2L_b\|\mu\|)^{k}\int_0^T\int_0^{s_1} \cdots\int_0^{s_{k-1}}\norm{u_1(s_{k})}{L^1(\Omega)}\dd s_{k}\dots \dd s_1\\
     \leq \ &\frac{(2L_b\|\mu\|T)^{k} }{k!}\max_{t\in[0,T]}\norm{u_1(t)}{L^1(\Omega)},
\end{align*}
where the last term is finite since  $u_1=I[0]\in C(0,T;L^1(\Omega))$ by Lemma \ref{lem:CL}. 
Thus $(u_k)_{k\in\N}$ is Cauchy and converges to some $u$ in $C(0,T;L^1(\Omega))$.  
Taking the limit as $k\to\infty$ inside the identity $u_{k+1}=I(u_k)$, we conclude from the continuity of $I$ that $u=I(u)$, i.e. $u$ is a renormalized entropy solution of \eqref{eq: conservationLawWithSourceTerm} with source term $g=-\L[b(\overline{u})]$ where $\overline{u}$ is the extension of $u$ to $M$ by $u_{|Q^c}=u^c$. Since $u_0$ and $g$ are bounded ($b$ is bounded and $\mathcal L: L^\infty \to L^\infty$), $u$ is bounded and an entropy solution of \eqref{eq: conservationLawWithSourceTerm} as defined in \eqref{eq: entropyFormulationBasicConservationLaw}: For all $k\in\R$ and $0\leq \varphi\in C^\infty_c([0,T)\times \R^d)$
\begin{equation}\label{eq: entropyFormulationForZeroOrderLWithNoR}
    \begin{split}
        &\,\int_Q -(u-k)^\pm\varphi_t - F^\pm(u,k)\cdot\nabla\varphi - \sgn^{\pm}(u-k)\L[b(\overline u)]\varphi\dd x\dd t\\
    \leq &\,  \int_{\Omega}(u_0-k)^{\pm}\varphi(0,\cdot)\dd x + L_f\int_{\Gamma}(\overline{u}^c - k)^{\pm}\varphi \dd \sigma(x)\dd t.
    \end{split}
\end{equation}

It remains to show that $u$ is in fact an entropy solution of \eqref{E} according to Definition \ref{def: entropySolutionVovelleMethod};  we need only show that inequality \eqref{eq: entropyFormulationForZeroOrderLWithNoR} implies \eqref{eq: entropyInequalityVovelleMethod} for all $r>0$, since $u$ is bounded and $u_{|Q^c}=u^c$ (this is implicit from the use of $\overline u$ in the $g$-term).
Take $r>0$, $k\in \R$, and $0\leq \varphi\in C^\infty_c([0,T)\times \R^d)$ satisfying
\begin{align*}
    (b(u^c)-b(k))^\pm\varphi =0 \qquad\text{in}\qquad Q^c.
\end{align*}
Since $\mu(\R^d)<\infty$, it immediately follows that for $(t,x)\in Q^c$,
\begin{align*}
    \L^{<r}[(b(u)-b(k))^\pm]\varphi
    =  \int_{|z|<r}\Big((b(u(\cdot,\cdot+z))-b(k))^\pm-(b(u)-b(k))^\pm\Big)\varphi\dd \mu(z)\geq 0,
\end{align*}
and then a convex inequality (cf. \cite[Corollary C.2]{AlEnJaMa23}) and symmetry of $\L^{<r}$ yields
\begin{align*}
     &\,\int_Q \sgn^\pm (u-k)\L^{<r}[b(u)]\varphi \dd x\dd t
    \leq  \int_Q \L^{<r}[(b(u)-b(k))^{\pm}]\varphi \dd x\dd t \\ 
    \leq&\,  \int_M \L^{<r}[(b(u)-b(k))^{\pm}]\varphi \dd x\dd t = \int_M ((b(u)-b(k))^{\pm})\L^{<r}[\varphi] \dd x\dd t.
\end{align*}
Hence, since $\L=\L^{<r}+\L^{\geq r}$, we conclude that 
\begin{align*}
    &\,\int_Q \sgn^\pm (u-k)\L
    [b(u)]\varphi \dd x\dd t 
    \leq  \int_M (b(u)-b(k))^\pm\L^{<r}[\varphi] \dd x\dd t+\int_Q\sgn^\pm (u-k)\L^{\geq r}[b(u)]\varphi\dd x\dd t .
\end{align*}
Inserting this in \eqref{eq: entropyFormulationForZeroOrderLWithNoR}, we get \eqref{eq: entropyInequalityVovelleMethod}.
\end{proof}
 
By the material developed in Section \ref{sec: stabiltyWithRespectToVariationsInL}, the unbounded case is a 
consequence of the former result:

\begin{theorem}[Existence for unbounded $\L$]\label{thm: existenceForUnboundedL}
Assume \eqref{Omegaassumption}--\eqref{u_0assumption}, and \eqref{muassumption2}.
Then there exists an entropy solution $u$ of \eqref{E}.
\end{theorem}
\begin{proof}
Fix $\varepsilon>0$ and introduce the 
 non-degenerate  Lévy operator
\begin{equation}\label{eq: mollifiedLevyOperator}
    \L_{\varepsilon}\coloneqq \L-\varepsilon(-\Delta)^{\frac{1}{2}},\quad \Longleftrightarrow\quad \dd\mu_{\varepsilon}(z)\coloneqq \dd \mu(z) + \varepsilon\frac{c_d}{|z|^{d+1}},
\end{equation}
where $c_d>0$ is some dimension-dependent constant. In view of Lemma \ref{lem: absolutelyContinuousMeasureHasNiceSymbol} (and Remark \ref{rem: thereExistProblematicLevyoperators}), it will be easier to prove existence for $\L_\varepsilon$: We do that first, and then extend the result by letting $\varepsilon\to 0$.
Define the bounded operators $\L_{n}\coloneqq \L^{\geq 1/n}_{\varepsilon}$ and let $(u_{n})_{n\in\N}$ denote the corresponding entropy solutions of \eqref{E} provided by Theorem \ref{thm: existenceForBoundedL}. We note two things. First, clearly $\L_n\to\L_\varepsilon$ in the sense of \eqref{muassumptionlim}. Secondly, by Lemma \ref{lem: absolutelyContinuousMeasureHasNiceSymbol}, the multipliers $(m_n)_{n\in\N}$ corresponding to $(\L_n)_{n\in\N}$ satisfy
\begin{equation}\label{eq: jointCoercivityConditionForExistenceProof}
    \liminf_{|\xi|,n\to\infty}m_n(\xi)=\infty,
\end{equation}
and so Corollary \ref{cor: jointCorecivityImpliesPrecompactnessOfb(u_n)} tells us that $(b(u_n))_{n\in \N}$ is precompact in $L^2(Q)$. Applying the stability result Proposition \ref{prop: stabilityOfEntropySolutionsWithRespectToPerturbationsInL}, we conclude that $(u_n)_{n\in\N}$ converges to the entropy solution $u_{\varepsilon}$ of \eqref{E} with Lévy operator $\L_{\varepsilon}$. As $\varepsilon>0$ was arbitrary, we have existence for all $(\L_{\varepsilon})_{\varepsilon>0}$ as defined in \eqref{eq: mollifiedLevyOperator}.

Next, pick a sequence $(\varepsilon_n)_{n\in\N}$ such that $\varepsilon_1>\varepsilon_2>\dots \to 0$, and redefine $u_n$ to be the entropy solution of \eqref{E} corresponding to the Lévy operator $\L_{\varepsilon_n}$. Again, it is easy to see that $\L_{\varepsilon_n}\to \L$ in sense of \eqref{muassumptionlim}, and we claim that the multipliers satisfy \eqref{eq: jointCoercivityConditionForExistenceProof} also here: Since $m_1\geq m_2\geq \dots\geq m$ where $m$ is the multiplier of (the original) $\L$, we obtain
\begin{equation*}
     \liminf_{|\xi|,n\to\infty}m_n(\xi)\geq \liminf_{|\xi|\to\infty}m(\xi)=\infty,
\end{equation*}
where the latter equality is \eqref{muassumption2}. As before, we get precompactness of $(b(u_n))_{n\in \N}\subset L^2(Q)$ from Corollary \ref{cor: jointCorecivityImpliesPrecompactnessOfb(u_n)}, and so applying the stability result Proposition \ref{prop: stabilityOfEntropySolutionsWithRespectToPerturbationsInL} we get the desired conclusion: $(u_n)_{n\in\N}$ converges to the entropy solution $u$ of \eqref{E} with Lévy operator $\L$. 
\end{proof}

\begin{remark}
     While we here mainly approximated a Lévy operator $\L$ using its truncations, there are other natural ways of doing so. Denoting the multiplier of $\L$ by $m$, it can be shown that $\frac{m}{1+\varepsilon m}$ is, for any $\varepsilon>0$, the multiplier of a bounded Lévy operator that converges to $\L$ in the sense of \eqref{muassumptionlim} as $\varepsilon\to 0$. (More generally, the same holds for $(g(0)-g(\varepsilon m))/(g'(0)\varepsilon)$ whenever $g\colon \R_+\to\R_+$ is completely monotonic \cite{BeFo75,Schilling2012}.) A small benefit of this approximation is that the joint coercivity condition \eqref{eq: jointCoercivityConditionForExistenceProof} is automatically satisfied whenever $m$ satisfies \eqref{muassumption2}, allowing one to skip the  $\varepsilon$-approximation argument  
     in the previous proof. That said, the approximation offers no improvement of our existence result.
     \end{remark}


\subsection*{Acknowledgments}
This research has been supported by the Research Council of Norway through different grants: All authors were supported by the Toppforsk (research excellence) grant agreement no. 250070 `Waves and Nonlinear Phenomena (WaNP)', OM is supported by grant agreement no. 301538 `Irregularity and Noise In Continuity Equations' (INICE), JE was supported by the MSCA-TOPP-UT grant agreement no. 312021, and ERJ is supported by grant agreement no. 325114 `IMod. Partial differential equations, statistics and data: An interdisciplinary approach to data-based modelling'. In addition, JE received funding from the European Union’s Horizon 2020 research and innovation programme under the Marie Sk{\l}odowska-Curie grant agreement no. 839749 `Novel techniques for quantitative behavior of convection-diffusion equations (techFRONT)'. Moreover, OM is supported by the French Agence Nationale de la Recherche (ANR) under grant number ANR-23-CE40-0015 (ISAAC).


\appendix


\section{Some results on Lévy operators and their multipliers}
\label{sec:LevyOperatorsInducingTranslationRegularity}

We use the following normalization of the Fourier transform 
\begin{equation}\label{eq: theFourierTransform}
    \hat{\phi}(\xi)\coloneqq \int_{\R^d}\phi(x)e^{-i\xi\cdot x}\dd x \quad\implies\quad  \phi(x)=\frac{1}{(2\pi)^d} \int_{\R^d}\hat{\phi}(\xi)e^{i\xi\cdot x}\dd \xi
\end{equation}
for $\phi\in L^2(\R^d)$. On the Fourier side, a symmetric Lévy operator $\L$ on $\R^d$ acts like a multiplication operator: Its multiplier is a continuous, nonpositive, and symmetric function $-m\colon\R^d\to\R$, given by
{\def\theequation{\ref{eq: definitionOfSymbolOfLevyOperator}}
\begin{align}
    m(\xi)\coloneqq  \int_{\R^d}\Big( 1 - \cos(\xi\cdot z)\Big)\dd\mu(z),
\end{align}
\addtocounter{equation}{-1}
}
\!\!\!\!where $\mu$ is the Lévy measure of $\L$. (By abuse of language, we also refer to $m$ as the multiplier of $\L$.) The integral in \eqref{eq: definitionOfSymbolOfLevyOperator} is well-defined because the integrand is bounded by $(\frac{1}{2}|\xi|^2|z|^2)\wedge 2$ and because of \eqref{muassumption}. The alleged connection between $\L$ and $m$ is easily proved:
\begin{lemma}\label{lem: fourierCharacterizationsOfLAndB}
   Let $\L$ and $m$ be as in the above setting. Then
    \begin{equation}\label{eq: mIsTheSymbol}
         \widehat{\L[\phi]}(\xi) = {-m}(\xi)\hat{\phi}(\xi), \qquad \forall \phi\in C^\infty_c(\R^d).
    \end{equation}
Moreover, with $H^\L_0(\Omega)$ as defined in \eqref{eq: definitionOfLevySpaceNorm}, we have 
    \begin{equation}    \label{eq: theLevyHilbertNormOnTheForuierSide}
 \int_{\R^d}\B[\phi,\phi]\dd x = \frac{1}{(2\pi)^{d}}\int_{\R^d}m(\xi)|\hat{\phi}(\xi)|^2\dd \xi, \qquad \forall \phi\in H^\L_0(\Omega).
    \end{equation}
In particular, the norm of $H^\L_0(\Omega)$ can be written $\norm{\phi}{H^{\L}_0}^2=(2\pi)^{-d}\int_{\R^d}\big(1+m(\xi)\big)|\hat{\phi}(\xi)|^2\dd \xi$.
\end{lemma}
\begin{proof}
For $\phi\in C^\infty_c(\R^d)$ we compute
\begin{align*}
     \widehat{\L[\phi]}(\xi) =&\, \int_{\R^d}\int_{\R^d}\Big(\tfrac{1}{2}\phi(x+z)+\tfrac{1}{2}\phi(x-z)-\phi(x)\Big)\dd \mu(z)e^{- i \xi\cdot x} \dd x\\
    =&\, \int_{\R^d}\int_{\R^d}\Big(\tfrac{1}{2}\phi(x+z)+\tfrac{1}{2}\phi(x-z)-\phi(x)\Big)e^{-i \xi\cdot x} \dd x\dd \mu(z)\\
    =&\, \int_{\R^d}\Big(\tfrac{1}{2}(e^{i \xi\cdot z} + e^{-i \xi\cdot z})-1\Big)\dd\mu(x)\hat{\phi}(\xi)\\
    =&\, \int_{\R^d}\Big(\cos(\xi\cdot z)-1\Big)\dd\mu(z)\hat{\phi}(\xi)={-m(\xi)\hat{\phi}(\xi)}
\end{align*}
where the use of Fubini's theorem is justified because
\begin{equation*}
 \int_{\R^d}|\tfrac{1}{2}\phi(x+z)+\tfrac{1}{2}\phi(x-z)-\phi(x)|\dd x\leq \bigg(\frac{|z|^2}{2}\norm{D^2\phi}{L^1(\R^d)}\bigg)\wedge\Big( 2\norm{\phi}{L^1(\R^d)}\Big).
\end{equation*}
For the second part, let $\phi\in H^\L_0(\Omega)$. Exploiting the identity $\frac{1}{2}|1-e^{i\xi\cdot z}|^2=1-\cos(\xi\cdot z)$ and Plancherel's Theorem, we have 
\begin{equation*}
\begin{split}
    \int_{\R^d}\B[\phi,\phi]\dd x = &\,\int_{\R^d}\int_{\R^d}\tfrac{1}{2}\big(\phi(x)-\phi(x+z)\big)^2\dd\mu(z)\dd x\\
    = &\,\frac{1}{(2\pi)^d}\int_{\R^d}\int_{\R^d}\tfrac{1}{2}|1-e^{i \xi\cdot z}|^2|\hat{\phi}(\xi)|^2\dd \xi \dd\mu(z)\\
    =&\,\frac{1}{(2\pi)^{d}}\int_{\R^d}\int_{\R^d}\Big(1-\cos(\xi\cdot z)\Big)\dd\mu(z)|\hat{\phi}(\xi)|^2\dd\xi
    =\frac{1}{(2\pi)^{d}}\int_{\R^d}m(\xi)|\hat{\phi}(\xi)|^2\dd \xi,    
\end{split}
\end{equation*}
where Fubini's theorem applies directly since the integrands are nonnegative. The last part of the lemma follows since $\phi\in H^{\L}_0(\Omega)$ is zero outside $\Omega$, and so $\norm{\phi}{L^2(\Omega)}^2=\norm{\phi}{L^2(\R^d)}^2=(2\pi)^{-d}\norm{\hat{\phi}}{L^2(\R^d)}^2$, which combined with the last identity gives \eqref{eq: theLevyHilbertNormOnTheForuierSide}.
\end{proof}

We now give a necessary and sufficient condition for when $H^\L_0(\Omega)$ compactly embeds in $L^2(\Omega)$.
{\def\thetheorem{\ref{prop: coerciveSymbolEqualsCompactEmbeddingOfLevySpaceInL2}}
\begin{proposition}
Let $\L$ be a symmetric Lévy operator, $m$ its multiplier, and let $H_0^{\L}(\Omega)$ be as in \eqref{eq: definitionOfLevySpaceNorm}.  Then the embedding $H^{\L}_0(\Omega)\hookrightarrow L^2(\Omega)$ is compact if, and only if, 
{\def\theequation{\ref{eq: criterionOnMeasureForCompactness}}
\begin{align}
    \liminf_{|\xi|\to\infty}m(\xi)=\infty,
\end{align}
which is precisely assumption \eqref{muassumption2}.
}
\end{proposition}
}\addtocounter{theorem}{-1}
\begin{proof}
Suppose first that $m$ satisfies \eqref{eq: criterionOnMeasureForCompactness}, and define $\rho\colon[0,\infty)\to[0,\infty)$ by
\begin{equation*}
    \rho(R)\coloneqq \inf_{|\xi|\geq R}m(\xi),
\end{equation*}
which must tend to infinity as $R\to\infty$.
Since $L^2(\Omega)$ is a metric space, we need only prove that a bounded sequence $(\phi_n)_{n\in\N}\subset H^{\L}_0(\Omega)$ admits a subsequence converging in $L^2(\R^d)$. For such a sequence, observe that we for any $n\in\N$, $y\in\R^d$ and any $R>0$ have by Plancherel's theorem
\begin{align*}
    \norm{\phi_n-\phi_n(\cdot + y)}{L^2(\R^d)}^2 = &\,\frac{1}{(2\pi)^{d}}\int_{\R^d}|1-e^{i\xi \cdot y}|^2|\hat{\phi}_n(\xi)|^2\dd\xi\\
    \leq &\,\frac{(R|y|)^2}{(2\pi)^{d}}\int_{|\xi|<R}|\hat{\phi}_n(\xi)|^2\dd\xi + \frac{4}{\rho(R)(2\pi)^{d}}\int_{|\xi|\geq R}m(\xi)|\hat \phi_n(\xi)|^2\dd\xi \\
    \leq &\,\bigg[(R|y|)^2 + \frac{4}{\rho(R)}\bigg]\norm{\phi_n}{H_0^{\L}(\Omega)}^2.
\end{align*}
Since we assume $\sup_{n} \norm{\phi_n}{H_0^{\L}(\Omega)}^2\eqqcolon C<\infty$, the previous calculation implies that
\begin{align*}
   \lim_{y\to0} \sup_{n} \norm{\phi_n-\phi_n(\cdot + y)}{L^2(\R^d)}^2 \leq\frac{4C}{\rho(R)}\to0
\end{align*}
as $R\to\infty$; we conclude that the family $(\phi_n)_{n\in\N}$ is equicontinuous with respect to translations. The family is also `equitight' as all functions are supported in $\Omega$. By the Fréchet–Kolmogorov theorem, $(\phi_n)_{n\in\N}$ then admits a sub-sequence converging in $L^2(\Omega)$.

Next, assume $m$ fails to satisfy \eqref{eq: criterionOnMeasureForCompactness}. Then there is a sequence $(\xi_n)_{n\in\N}\subset \R^d$ such that $\lim_{n\to \infty}|\xi_{n}|=\infty$ and $\sup_{n}m(\xi_n)\leq c_0$ for some finite $c_0$. We shall construct a bounded sequence $(\phi_n)_{n\in\N}\subset H^{\L}_0(\Omega)$ for which there is no sub-sequence converging in $L^2(\Omega)$. Without loss of generality, we assume the corresponding measure $\mu$ of $\L$ to be supported in the unit ball of $\R^d$; after all, the space induced by $\L$ is equivalent to that of $\L^{<1}=\L - \L^{\geq 1}$ because $\L^{\geq 1}$ is a zero order operator. We then see from its definition that $m$ is an analytic function whose double derivative in any direction $\nu\in \mathbb{S}^{d-1}$ satisfies
\begin{align*}
    \bigg|\frac{\partial^2}{\partial s^2}m(\xi+s\nu)\bigg|_{s=0}\leq \int_{|z|\leq 1}|z|^2d\mu(z)\eqqcolon c_1.
\end{align*}
Without loss of generality, assume also that $\Omega$ includes the origin in its interior and select a non-zero, radially symmetric function $\varphi\in C^\infty_c(\Omega)$. Consider the sequence $(\phi_n)_{n\in\N}\subset L^2(\Omega)$ defined by
\begin{align*}
    \phi_n(x)\coloneqq \cos(\xi_n\cdot x)\varphi(x), \quad \implies \quad \widehat{\phi_n}(\xi) = \tfrac{1}{2}(\hat{\varphi}(\xi+\xi_n)+\hat{\varphi}(\xi-\xi_n)),
\end{align*}
where $(\xi_n)_{n\in\N}$ is the sequence from before. Using \eqref{eq: theLevyHilbertNormOnTheForuierSide}, the symmetry of $m$ and $\hat\varphi$, and performing two changes of variables we compute 
\begin{align*}
(2\pi)^d\int_{\R^d}\B[\phi_n,\phi_n]\dd x =&\, \int_{\R^d}m(\xi)\frac{|\hat{\varphi}(\xi+\xi_n)+\hat{\varphi}(\xi-\xi_n)|^2}{4}\dd \xi\\
   \leq &\, \int_{\R^d}m(\xi + \xi_n)|\hat{\varphi}(\xi)|^2\dd\xi\\
    \leq &\, \int_{\R^d}\Big(c_0 + \nabla m(\xi_n)\cdot \xi + \frac{c_1}{2}|\xi|^2\Big) |\hat{\varphi}(\xi)|^2\dd\xi,
\end{align*}
where the last inequality follows by Taylor expansion. As $\varphi$ is smooth and rapidly decaying, so is $\hat{\varphi}$ and thus the last integral is well-defined. Moreover, since $\varphi$ is symmetric (about zero) so is $\hat{\varphi}$ and consequently $\int_{\R^d}\nabla m(\xi_n)\cdot\xi|\hat{\varphi}(\xi)|^2\dd \xi=0$. We conclude
\begin{align*}
    \sup_{n}\norm{\phi_n}{H^\L_0(\Omega)}^2\leq \frac{1}{(2\pi)^d}\int_{\R^d}\Big(1+c_0 + \frac{c_1}{2}|\xi^2|\Big)|\hat{\varphi}(\xi)|^2\dd \xi<\infty,
\end{align*}
 that is, the sequence $(\phi_n)_{n\in\N}$ is bounded in $H^{\L}_0(\Omega)$. However, no subsequence of $(\phi_n)_{n\in\N}$ converges in $L^2(\Omega)$ since no subsequence of $(\hat{\varphi}(\cdot-\xi_n))_{n\in\N}$ converges in $L^2(\R^d)$. This concludes the proof.
\end{proof}

\begin{lemma}\label{lem: zeroOrderLevyOperatorsHaveBoundedSymbols}
    Let $\L$ be a symmetric Lévy operator on $\R^d$. Then $\L$ is zero order (its measure $\mu$ finite) if, and only if, the corresponding multiplier $m$ is bounded. We then have the following inequality
    \begin{equation*}
       \mu(\R^d)\leq \norm{m}{L^\infty(\R^d)}\leq 2\mu(\R^d).
    \end{equation*}
\end{lemma}
\begin{proof}
    From its definition, we see that $m(\xi)\leq 2\mu(\R^d)$ when the latter quantity is finite. To prove the other direction, let $V_d(R)=V_d(1)R^d$ denote the volume of a ball in $\R^d$ of radius $R$, and note for $R>0$ that
\begin{align*}
     \norm{m}{L^{\infty}} \geq&\,  \int_{|\xi|\leq R}m(\xi)\frac{\dd \xi}{V_d(R)} =  \int_{|\xi|\leq R}\int_{\R^d}\Big(1-\cos(\xi\cdot z)\Big)\dd\mu(z)\frac{\dd \xi}{V_d(R)}\\
      =&\, \int_{\R^d} \int_{|\xi|\leq R}\Big(1-\cos(\xi\cdot z)\Big)\frac{\dd \xi}{V_d(R)}\dd\mu(z) =\int_{\R^d} \int_{|\xi|\leq 1}\Big(1-\cos(\xi\cdot Rz)\Big)\frac{\dd \xi}{ V_d(1)}\dd\mu(z)
\end{align*}
where we employed Fubini's theorem. For each $|z|>0$ the value of the inner integral at the end lies in $[0,2]$ and converges (by the Riemann-Lebesgue lemma) to $1$ as $R\to\infty$. Thus, the final integral converges to $\mu(\R^d\setminus \{0\})=\mu(\R^d)$ when $R\to\infty$.
\end{proof}

\begin{lemma}\label{lem: absolutelyContinuousMeasureHasNiceSymbol}
    Let $\L$ be a symmetric Lévy operator on $\R^d$ with measure $\mu$ and let $\mu_a$ denote the absolutely continuous part of $\mu$. If $\mu_a(\R^d)=\infty$ then $\mu$ satisfies \eqref{muassumption2}. Moreover, the sequence of multipliers $(m_n)_{n\in\N}$ corresponding to the truncations $(\L^{\geq(1/n)})_{n\in\N}$ satisfies 
    \begin{equation}\label{eq: jointCoercivityConditionInAppendix}
\liminf_{|\xi|,n\to\infty}m_n(\xi)=\infty,
    \end{equation}
which is the coercivity condition \eqref{eq: theGeneralCoercivityConditionOfSymbol} from Proposition \ref{prop: uniformTranslationContinuityForASequenceOfFiniteEnergies}.
\end{lemma}
\begin{proof}
It suffices to prove \eqref{eq: jointCoercivityConditionInAppendix} since this also implies \eqref{muassumption2} by $\sup_{n}m_n(\xi)=m(\xi)$.
For $n\in \N$ the measure $\mu_a$ is finite on $\{|z|\geq \frac{1}{n}\}$ and so we can do the following splitting of the integral
\begin{align*}
    m_n(\xi)\geq &\, \int_{|z|\geq \frac{1}{n}} \Big(1-\cos(\xi\cdot z)\Big)\dd\mu_a(z) \\
    = &\, \mu_a\big(\{|z|\geq \tfrac{1}{n}\}\big) - \int_{|z|\geq \frac{1}{n}} \cos(\xi \cdot z)\frac{\dd\mu_a}{\dd z}(z)\dd z\\  \to&\,\mu_a\big(\{|z|\geq \tfrac{1}{n}\}\big)-0
\end{align*}
as $|\xi|\to\infty$ by the Riemann--Lebesgue lemma since  $\frac{\dd\mu_a}{\dd z}1_{|z|\geq \tfrac{1}{n}}\in L^1(\R^d)$ by \eqref{muassumption}.  Thus, we get 
\begin{equation*}
    \lim_{n\to\infty}\liminf_{|\xi|\to\infty}m_n(\xi)\geq \lim_{n\to\infty}\mu_a\big(\{|z|\geq \tfrac{1}{n}\}\big) =\mu_a\big(\R^d\setminus\{0\}\big)=\mu_a\big(\R^d\big)=\infty. \qedhere
\end{equation*}
By Remark \ref{remark: aSimplificationToTheComplicatedCorecivityCondition}, this last inequality implies \eqref{eq: jointCoercivityConditionInAppendix} since $m_1\leq m_2\leq \dots$.
\end{proof}
Together, the two previous lemmas imply:
\begin{corollary}\label{cor: allAbsolutelyContinuousMeasuresAreIncludedInOurExistenceResult}
    Let $\mu$ be a symmetric and absolutely continuous Lévy measure on $\R^d$. Then either $\mu$ is finite or it satisfies \eqref{muassumption2}, but not both.
\end{corollary}

\begin{lemma}[Some examples of badly behaving Lévy operators]\label{lem: examplesOfBadLevyOperators}
    Let $\L$ be a symmetric Lévy operator on $\R^d$ with corresponding Lévy measure $\mu$ and multiplier $m$. The following two scenarios can occur:
    \begin{enumerate}[(a)]
        \item The measure $\mu$ may fail both assumptions ``$\mu(\R^d)<\infty$'' and \eqref{muassumption2} of Theorem \ref{thm: mainResultExistence}. Equivalently (by Lemma \ref{lem: absolutelyContinuousMeasureHasNiceSymbol}), the multiplier may be such that
        \begin{equation*}
            \liminf_{|\xi|\to\infty}m(\xi)< \limsup_{|\xi|\to\infty}m(\xi)=\infty.
        \end{equation*}
        \item When \eqref{muassumption2} is satisfied, the multipliers $(m_n)_{n\in\N}$ corresponding to the truncations $(\L^{\geq 1/n})_{n\in\N}$ may still  fail to satisfy \eqref{eq: jointCoercivityConditionInAppendix}. In this case, a bounded sequence $(\phi_n)_{n\in\N}\subset L^2(\Omega)$ exists such that
        \begin{equation*}
            \sup_{n}\int_{\R^d}\B^{\geq 1/n}[\phi_n,\phi_n]\,\dd x<\infty,
        \end{equation*}
        but for which no subsequence converges.
    \end{enumerate}
\end{lemma}
\begin{proof}
    For simplicity we restrict our attention to $d=1$, although the following can easily be generalized to $d>1$ (using for example radially symmetric measures). For  part (a),  
    we set 
    \begin{align*}
    \mu = \sum_{k=1}^{\infty}\frac{\delta(x+\tfrac{1}{2^k}) + \delta(x-\tfrac{1}{2^k})}{2} \quad\implies\quad \int_{\R}(|z|^2\wedge 1)\dd\mu(z) = \sum_{k=1}^\infty\frac{1}{4^k}= \frac{1}{3}. 
\end{align*}
Clearly $\mu$ satisfies \eqref{muassumption} and is thus a valid Lévy measure. It is not finite, $\mu(\R)=\infty$, and it fails to satisfy \eqref{muassumption2}: Evaluating the corresponding multiplier $m$ at $\xi = \pi2^n$ with $n\in\N$, we find
\begin{align*}
    m(\pi 2^n) =&\, \int_{\R}\Big(1-\cos(\pi 2^n z)\Big)\dd \mu(z) = \sum_{k=1}^\infty \Big(1-\cos(\pi 2^{n-k})\Big)\\
    =&\,\sum_{k=n}^\infty \Big(1-\cos(\pi 2^{n-k})\Big)\leq 
     \frac12\sum_{\ell=0}^\infty\Big(\frac{\pi}{2^{\ell}}\Big)^2=\frac{2}{3}\pi^2,
\end{align*}
where we used that $\cos(\pi 2^{n-k})=1$ if $k<n$ and the bound $1-\cos(s)\leq \tfrac{1}{2}s^2$.

For  part (b)  
we similarly set
\begin{align*}
    \mu = \sum_{k=1}^{\infty}2^{k}\bigg(\frac{\delta(x+\tfrac{1}{2^k}) + \delta(x-\tfrac{1}{2^k})}{2}\bigg)\quad\implies\quad \int_{\R}(|z|^2\wedge1)\dd\mu(z) = \sum_{k=1}^\infty\frac{1}{2^k} =1,  
\end{align*}
which is also a valid Lévy measure. Moreover $\mu$ satisfies \eqref{muassumption2}: For any  $\xi\in\R\setminus[-2,2]$ pick $n\in \N$ such that $|\xi|/2^n\in [1,2]$ and note that
\begin{align*}
    m(\xi) = \sum_{k=1}^\infty 2^{k}\Big(1-\cos\Big(\frac{\xi}{2^k} \Big)\Big)\geq 2^n\Big(1-\cos(1)\Big),
\end{align*}
where we used that $\cos(s)\leq \cos(1)$ for $s\in [1,2]$.
Thus, $\liminf_{|\xi|\to\infty}m(\xi)=\infty$. The sequence of truncated measures and multipliers are, up to a relabeling after disposing of repeating elements, given by
\begin{align*}
    \mu_n =&\, \sum_{k=1}^{n}2^{k}\bigg(\frac{\delta(x+\tfrac{1}{2^k}) + \delta(x-\tfrac{1}{2^k})}{2}\bigg), &
    m_n(\xi) = &\, \sum_{k=1}^{n}2^{k}\Big(1-\cos\Big(\frac{\xi}{2^k}\Big)\Big).
\end{align*}
Setting $\xi_n \coloneqq \pi 2^{n+1}$ we observe that $m_n(\xi_n)=0$, and so $(m_n)_{n\in\N}$ fail to satisfy \eqref{eq: jointCoercivityConditionInAppendix}. The construction of the mentioned sequence $(\phi_n)_{n\in \N}$ is analogous to the construction from Proposition \ref{prop: coerciveSymbolEqualsCompactEmbeddingOfLevySpaceInL2} and so we skip it. 
\end{proof}


\section{A technical lemma}

\begin{lemma}\label{lem: theAverageAlmostCommuteWithVShapedFunction}
Assume $R>0$, $\kappa$ is a Borel probability measure on $[-R,R]$,  $h\colon [-R,R]\to[0,\infty)$ is Lipschitz with Lipschitz-constant $L$, non-increasing on $[-R,0]$, non-decreasing on $[0,R]$, and $h(0)=0$. The two $\kappa$-means $\ol s\coloneqq \int_{[-R,R]} s\dd \k(s)$ and $\ol{h}\coloneqq \int_{[-R,R]} h(s)\dd \k(s)$ then satisfy 
\begin{align*}
    h(\overline{s})^2\leq LR\overline{h}.
\end{align*}
\end{lemma}
\begin{remark}
This lemma is comparable to (and can be viewed as a generalization of) Lemma 4.5 in \cite{EyGaHe00} which states that for every $c<d$ and  monotone Lipschitz $b$,
$$
|b(d)-b(c)|^2\leq 2L_b\int_{c}^{d}\big(b(x)-b(c)\big)\dd x.
$$
\end{remark}

\begin{proof}[Proof of Lemma \ref{lem: theAverageAlmostCommuteWithVShapedFunction}]
Without loss of generality we assume both $\overline{s}>0$ and $h(\overline{s})>0$. 
Observe first that since $h(\overline{s})\leq h(s)$ for $s\in [\overline{s},R]$, we get
\begin{equation}\label{eq: technicalEquationForTechnicalLemma1}
    \begin{split}
        h(\overline{s})-\overline{h}=&\,\Big(\int_{[-R,\overline{s})}+\int_{[\overline{s},R]}\Big)\big(h(\overline{s})-h(s)\big)\dd \k(s) 
    \leq L\int_{[-R,\overline{s})} (\overline{s}-s)\dd \k(s)+0.
    \end{split}
\end{equation}
Using that $\bar s-s$ has $\kappa$-mean zero and that $h(\overline{s})\leq L\overline{s}$ (since $h(0)=0$), we further compute 
\begin{equation}
    \begin{split}\label{eq: technicalEquationForTechnicalLemma2}
    &L\int_{[-R,\overline{s})}(\overline{s}-s)\dd \k(s)= L\int_{[\overline{s},R]} (s-\overline{s}) \dd \k(s)
    \leq (LR-h(\overline{s}))\int_{[\overline{s},R]}  \dd \k(s)\\
    &\leq \bigg(\frac{LR}{h(\overline{s})}-1\bigg)\int_{[\overline{s},R]} h(s)\dd \k(s)
    \leq \bigg(\frac{LR}{h(\overline{s})}-1\bigg)\overline{h}.
\end{split}
\end{equation}
Combining \eqref{eq: technicalEquationForTechnicalLemma1} with \eqref{eq: technicalEquationForTechnicalLemma2} gives the result.
\end{proof}



\end{document}